\documentclass[11pt]{amsart}   	
\usepackage[margin = 1in]{geometry}
									
\usepackage{amsmath,amssymb, amsfonts,amsthm}
\usepackage{tensor}
\usepackage[dvipsnames]{xcolor}
\usepackage{hyperref}
\usepackage{mathtools}
\usepackage{tikz}
\usepackage{tikz-cd}
\usetikzlibrary{decorations.markings}
\usepackage{caption}
\usepackage{enumitem}
\usepackage{array}
\usepackage{makecell}
\setcellgapes{4pt}

\newtheorem{theorem}{Theorem}[section]
\newtheorem{definition-theorem}[theorem]{Definition-Theorem}
\newtheorem{lemma}[theorem]{Lemma}
\newtheorem{corollary}[theorem]{Corollary}
\newtheorem{proposition}[theorem]{Proposition}
\newtheorem{conjecture}[theorem]{Conjecture}

\theoremstyle{definition}
\newtheorem{definition}[theorem]{Definition}
\newtheorem{example}[theorem]{Example}

\newtheorem{remark}[theorem]{Remark}
\newtheorem{notation}[theorem]{Notation}
\newtheorem{question}[theorem]{Question}

\newcommand{\lperp}[1]{\prescript{\perp}{}{#1}}

\newcommand{\add}{\mathsf{add}}
\newcommand{\Filt}{\mathsf{Filt}}
\DeclareMathOperator{\Fac}{\mathsf{Fac}}
\DeclareMathOperator{\Gen}{\mathsf{Gen}}
\DeclareMathOperator{\wide}{\mathsf{wide}}
\DeclareMathOperator{\tors}{\mathsf{tors}}
\newcommand{\mods}{\mathsf{mod}}
\newcommand{\Hom}{\mathrm{Hom}}

\newcommand{\jirr}{\mathsf{cj}\text{-}\mathsf{irr}}
\newcommand{\mirr}{\mathsf{cm}\text{-}\mathsf{irr}}

\newcommand{\nuc}{\mathsf{binuc}}
\newcommand{\proj}{\mathsf{proj}}

\newcommand{\undim}{\underline{\mathrm{dim}}}

\newcommand{\T}{\mathcal{T}}
\newcommand{\F}{\mathcal{F}}
\newcommand{\W}{\mathcal{W}}
\newcommand{\U}{\mathcal{U}}
\newcommand{\V}{\mathcal{V}}

\newcommand{\A}{\mathcal{A}}
\newcommand{\K}{\mathcal{K}}
\newcommand{\C}{\mathcal{C}}
\newcommand{\pop}{\mathrm{pop}}

\newcommand{\tor}{\mathsf{T}}
\renewcommand{\L}{\mathcal{L}}

\newcommand{\FW}{\leq_{\mathrm{fss}}}
\newcommand{\WI}{\leq_{\mathrm{NI}}}
\newcommand{\TF}{\mathrm{TF}}
\newcommand{\gTF}{g\text{-}\mathrm{TF}}
\newcommand{\FWjoin}{\vee_{\mathrm{fss}}}
\newcommand{\FWmeet}{\wedge_{\mathrm{fss}}}
\newcommand{\FWneq}{\lneq_{\mathrm{fss}}}
\newcommand{\WIjoin}{\vee_{\mathrm{NI}}}
\newcommand{\WImeet}{\wedge_{\mathrm{NI}}}
\newcommand{\WIneq}{\lneq_{\mathrm{NI}}}

\renewcommand{\int}{\mathsf{int}}

\newcommand{\covered}{{\,\,<\!\!\!\!\cdot\,\,\,}}
\newcommand{\FWcover}{<\!\!\!\!\cdot_{\mathrm{fss}\,\,\,}}
\newcommand{\WIcover}{<\!\!\!\!\cdot_{\mathrm{NI}\,\,\,}}

\usepackage{soul}

\author{Eric J. Hanson}

\address{D\'epartement de Math\'ematiques, LACIM, Universit\'e du Qu\'ebec \`a Montr\'eal and\newline\indent D\'epartement de Math\'ematiques, Universit\'e de Sherbrooke}
\email{eric.james.hanson@usherbrooke.ca}

\subjclass[2020]{05E10, 16G20, 18E40, 52C99 (primary), 06A07, 06D75 (secondary)}

\keywords{$g$-vector fan, facial weak order, $\tau$-tilting theory, 2-term silting, semidistributive lattices}

\title{A facial order for torsion classes}

\date{\today}

\begin{document}

\begin{abstract}
	We generalize the ``facial weak order'' of a finite Coxeter group to a partial order on a set of intervals in a complete lattice. We apply our construction to the lattice of torsion classes of a finite-dimensional algebra and consider its restriction to intervals coming from stability conditions. We give two additional interpretations of the resulting ``facial semistable order'': one using cover relations, and one using Bongartz completions of 2-term presilting objects. For $\tau$-tilting finite algebras, this allows us to prove that the facial semistable order is a semidistributive lattice. We then show that, in any abelian length category, our new partial order can be partitioned into a set of completely semidistributive lattices, one of which is the original lattice of torsion classes.
\end{abstract}

\maketitle

\tableofcontents


\section{Introduction}

Let $W$ be a finite Coxeter group. Then the weak (Bruhat) order on $W$ can be seen as a partial order on the vertices of the $W$-permutahedron. In \cite{KLNPS} (type A) and \cite{PR} (arbitrary type), this partial order was extended to all faces of the $W$-permutahedron. The resulting partial order was termed the \emph{facial weak order} in \cite{DHP}.

In \cite{DHP}, the authors consider the dual picture; that is, they consider the facial weak order as a partial order on the cones of the Coxeter fan associated to $W$. This results in a global description of the facial weak order: given a cone $C$ in the fan, the set of maximal-dimensional cones which contain $C$ as a face form an interval $[m_C,M_C]$ in the weak order. Two cones $C$ and $D$ then satisfy $C\leq D$ in the facial weak order if and only if $m_C \leq m_D$ and $M_C \leq M_D$ in the traditional weak order. (This can also be reformulated in terms of minimal and maximal coset representatives.)

Now let $\mathcal{H}$ be a simplicial central hyperplane arrangements. Then the ``poset of regions'' of $\mathcal{H}$ (with any choice of base region) is known to be a lattice \cite[Theorem~3.4]{BEZ}, and is isomorphic to the weak order when $\mathcal{H}$ is a hyperplane arrangement. Thus the global description of \cite{DHP} can be used to define a ``facial weak order'' for the poset of regions. This was constructed in \cite{DHMP}, where it was also shown that one can describe the facial weak order of a hyperplane arrangement via its cover relations (which is how the original facial weak order was defined) and via root inversion sets. These descriptions were used in \cite{DHP} (Coxeter arrangement case) and \cite{DHMP} (arbitrary case) to prove the the facial weak order of any simplicial central hyperplane arrangement is a lattice which contains the poset or regions (or weak order) as a sublattice. This was also shown for the symmetric groups in \cite{KLNPS}.

One can also consider the restriction of the facial weak order of a Coxeter group to the faces of (generalized) associahedra, and indeed this consideration appears already in \cite{PR} in type A. Each generalized associahedron is combinatorially dual to a Cambrian fan \cite{RS}, and the restriction of the weak order to the vertices of the generalized associahedron is a Cambrian lattice \cite{reading_cambrian}. In \cite{DHP}, it is shown that any lattice congruence on the weak order of a Coxeter group extends to a lattice congruence of the facial weak order, thus as a special case yielding a ``facial Cambrian lattice'' structure on the cones of the Cambrian fan.

Cambrian fans play an important role in the theory of cluster algebras. Indeed, the $g$-vectors of a finite-type cluster algebra form a complete simplicial fan (the ``$g$-vector fan'') in $\mathbb{R}^n$, see \cite{FZ4,DWZ,FZ_Y}. This fan has also appeared in the contexts of tropical cluster $\mathcal{X}$-varieties \cite{FG}, cluster scattering \cite{GHHK}, and stability conditions \cite{bridgeland}. For acyclic cluster algebras, the $g$-vector fan is known to be combinatorially isomorphic to the corresponding Cambrian fan, see \cite{HLT,HPS}. We also refer readers to \cite{AHHL,BMCLDMTY,PPPP} and the references therein for a more detailed history of the realization problem for fans associated to cluster algebras.

The $g$-vectors of finite-type cluster algebras also admit an interpretation, and generalization, via the silting theory of finite-dimensional associative algebras over fields, see e.g. \cite{AIR,DIJ}. Indeed, the $g$-vector fans associated to finite-dimensional algebras have been of recent interest to many authors, see e.g. \cite{PY, AHIKM,mizuno_shards} and the references therein. In this case, the cones of the $g$-vector fan correspond to ``2-term presilting complexes'', or alternatively to ``support $\tau$-tilitng pairs''. Again, the maximal cones of the fan are equipped with a natural partial order: that coming from the mutation graph of support $\tau$-tilting pairs \cite{AIR}, or alternatively the inclusion relation on functorially finite torsion classes \cite{DIJ}. For hereditary (resp. preprojective, see \cite{mizuno}) algebras of Dynkin type, the facial weak order on the Cambrian (resp. Coxeter) fan can thus be interpreted as a lattice structure on all support ($\tau$-)rigid pairs, not just the complete ones.

The goal of this paper is to define and study an analog of the facial weak order on the $g$-vector fan associated to an arbitrary finite-dimensional algebra. (Note that the $g$-vector fan will not be complete in general.) More generally, we define a partial order on certain intervals of an arbitrary complete lattice (called the ``binuclear interval order'', see Definition~\ref{def:nuc_order}). While this partial order is not a lattice in general (Example~\ref{ex:not_lattice}), it does admit a universal formula for those joins and meets which do exist (Lemma~\ref{lem:lattice1}). Examples where the lattice property does hold are established in \cite{DHP,DHMP} and Section~\ref{sec:lattice} of the present paper, and it remains an open question (Question~\ref{ques:lattice}) to fully characterize when the binuclear interval order is a lattice.

In Sections~\ref{sec:tors_background} and~\ref{sec:semistable}, we construct and interpret the binuclear interval orders of lattices of torsion classes. In particular, we introduce the ``facial semistable order'' (Definition~\ref{def:semistable}) as the restriction of the binuclear interval order to the intervals corresponding to ``semistable torsion classes''. This yields a partial order on a partition of $\mathbb{R}^n$, and it remains an open question to fully understand its relationship with the usual product order. See Remark~\ref{rem:nonnegative}.

Restricting further to the ``functorially finite'' torsion classes, we obtain a partial order on the cones of the $g$-vector fan. This restriction can then understood as a partial order on 2-term presilting complexes as follows. (See Sections~\ref{sec:tors_background} and~\ref{sec:semistable} for an explanation of the notation.)

\begin{proposition}[Proposition~\ref{prop:completions}]\label{prop:main}
	Let $U$ and $V$ be 2-term presilting complexes over a finite-dimensional algebra. Then $U \leq V$ in the facial semistable order if and only if $\mathrm{Gen}(H_0(U)) \subseteq \mathrm{Gen}(H_0(V))$ and $\lperp{(H^{-1}(\nu U))} \subseteq \lperp{(H^{-1}(\nu V))}$.
\end{proposition}

We study the cover relations of the facial semistable order in Section~\ref{sec:covers}. In particular, we prove our first main theorem.

\begin{theorem}[Theorem~\ref{thm:covers}, simplified]\label{thm:mainA}
	Let $A$ be a finite-dimensional algebra, and let $C$ and $D$ be cones in the $g$-vector fan. Then $C$ is covered by $D$ in the facial semistable order if and only if one of the following holds:
	\begin{enumerate}
		\item $C$ is a codimension-1 face of $D$ and the corresponding 2-term presilting complexes have the same Bongartz completion.
		\item $D$ is a codimension-1 face of $C$ and the corresponding 2-term presilting complexes have the same co-Bongartz completion.
	\end{enumerate}
\end{theorem}

We also consider cover relations in the portion of the facial semistable which lies outside the $g$-vector fan. As a consequence, we obtain the following fact which is interesting in its own right.

\begin{theorem}[Theorem~\ref{thm:covers2} and Remark~\ref{rem:covers2}, simplified]\label{thm:mainB}
	Let $A$ be a finite-dimensional algebra, and let $\theta, \eta \in \mathbb{R}^n$ both lie in the $g$-vector fan. If the interior on the line segment connecting $\theta$ and $\eta$ lies completely in the complement of the $g$-vector fan, then the cones containing $\theta$ and $\eta$ are not related under the facial semistable order.
\end{theorem}

This characterization allows us to deduce that the restriction of the facial semistable order to the $g$-vector fan has the same \emph{undirected} Hasse graph as the usual inclusion order on cones. See Remark~\ref{rem:inclusion}.

In Section~\ref{sec:lattice}, we restrict our attention to algebras with complete $g$-vector fans. We then prove our third main result.

\begin{theorem}[Theorem~\ref{thm:lattice}]\label{thm:mainC}
	Let $A$ be a finite-dimensional algebra with a complete $g$-vector fan. Then the facial semistable order (and equivalently the binuclear interval order) associated to $A$ is a lattice.
\end{theorem}

In Section~\ref{sec:semidistributive}, we study the behavior of properties related to semidistributivity when one passes from a complete lattice to its binuclear interval order. In particular, we show that semidistributivity is inherited by the facial semistable order of algebras with complete $g$-vector fans (Corollary~\ref{cor:alg_semidistributive}). This mirrors the behavior observed in \cite{DHMP} for posets of regions. It remains an open question to give a complete description of which of the lattice properties considered in this section pass to binuclear interval orders, see Question~\ref{ques:semidistributive}.

In Section~\ref{sec:recover_tors}, we prove our last main result. For algebras with complete $g$-vector fans, this in particular recovers the lattice of torsion classes as a sublattice of the facial semistable order.

\begin{theorem}[Theorem~\ref{thm:sublattice}]\label{thm:mainD}
	Let $\A$ be an abelian length category. Then the binuclear interval order of the lattice of torsion classes of $\A$ can be partitioned into a set of completely semidistributive lattices. Moreover, one of these lattices is isomorphic to the original lattice of torsion classes.
\end{theorem}

It remains an open problem to give a lattice-theoretic proof and generalization of Theorem~\ref{thm:mainD}, see Remark~\ref{rem:lattice_recover_tors}.

We conclude by giving two examples in Section~\ref{sec:examples}.


\subsection*{Acknowledgements}

The author is thankful to Pierre-Guy Plamondon for hosting them for a research stay at Universit\'e de Versailles Saint-Quentin and for numerous insightful conversations related to this project. A large portion of this work was also completed while the author was a visiting fellow in the program ``Representation Theory: Combinatorial Aspects and Applications'' held at the Centre for Advanced Study at the Norwegian Academy of Science and Letters. The author thanks these institutions for their hospitality. They also thank Emily Barnard, Colin Defant, Peter J{\o}rgensen, Nathan Reading, Hugh Thomas, and Emine Y{\i}ld{\i}r{\i}m for helpful discussions. The author was partially supported by the Canada Research Chairs program (CRC-2021-00120) and NSERC Discovery Grants (RGPIN-2022-03960 and RGPIN/04465-2019).


\section{The binuclear interval order}\label{sec:lattice_background}

In this section, we recall background information about lattices and introduce the binuclear interval order (Definition~\ref{def:nuc_order}).

Let $\L = (\L,\leq)$ be a poset. For $S \subseteq \L$, we denote by $\bigwedge S = \bigwedge_\L S$ the \emph{meet}, or greatest lower bound, of $S$ in $\L$, if it exists. Similarly, we denote by $\bigvee S = \bigvee_\L S$ the \emph{join}, or least upper bound, if it exists. The poset $\L$ is called a \emph{lattice} (resp. \emph{complete lattice}) if $\bigwedge\{x,y\} =: x \wedge y$ and $\bigvee\{x,y\} =: x \vee y$ (resp. $\bigwedge S$ and $\bigvee S$) exist for all $x, y \in \L$ (resp. for all $S \subseteq \L$). Note that any finite lattice is necessarily complete, and that any complete lattice contains a unique minimum $\hat{0} = \bigvee \emptyset$ and maximum $\hat{1} = \bigwedge \emptyset$.

Given $x, y \in \L$, we say that $y$ \emph{covers} $x$, or that $x \covered y$ is a \emph{cover relation}, if $x \lneq y$ and there does not exist $z \in \L$ with $x \lneq z \lneq y$. The \emph{Hasse quiver} of $\L$ is then defined as the directed graph with a vertex for every $x \in \L$ and an arrow $y \rightarrow x$ whenever $x \covered y$. Throughout this paper, we will typically illustrate examples using Hasse quivers. Note also that it is well-known that finite posets are uniquely determined by their Hasse quivers.

Suppose for the remainder of this section that $\L$ is a complete lattice. We consider the following operators on $\L$.

\begin{definition}\label{def:pop}
	Let $x \leq y \in \L$. We denote
	\begin{eqnarray*}
		\pop_x(y) &:=& y \wedge \left(\bigwedge \{z \in \L \mid x \leq z \covered y\}\right)\\
		\pop^y(x) &:=& x \vee \left(\bigvee \{z \in \L \mid x \covered z \leq y\}\right).
	\end{eqnarray*}
\end{definition}

\begin{remark}
	Note that $\pop_{\hat{0}}(y) = y \wedge \left(\bigwedge \{z \in \L \mid z \covered y\}\right)$ and $\pop^{\hat{1}}(x) = x \vee \left(\bigvee \{z \in \L \mid x \covered z\}\right)$. These are sometimes called the ``pop-stack sorting operators'', and have been of recent interest in the field of dynamical combinatorics. See e.g. \cite[Section~1.2]{DW} and the references therein. They also play a key role in defining alternative partial orders on completely semidistributive lattices, see e.g. \cite[Section~9-7.4]{reading_lattice}, \cite{muhle}, \cite[Section~4.4]{enomoto_lattice}, and \cite[Section~7]{BaH2}.
\end{remark}

By an \emph{interval} in $\L$, we will always mean a closed interval, i.e., a subset of the form $[x,y] = \{z \in \L \mid x \leq z \leq y\}$ for some $x \leq y \in \L$. For $I \subseteq \L$ an interval, we denote by $I^-$ and $I^+$ the minimum and maximum elements of $I$, respectively. We then consider the following definition.

\begin{definition}\label{def:nuclear}
	Let $I \subseteq \L$ be an interval.
	\begin{enumerate}
		\item We say that $I$ is \emph{nuclear} if $I^- = \pop_{I^-}(I^+)$.
		\item We say that $I$ is \emph{conuclear} if $I^+ = \pop^{I^+}(I^-)$.
		\item We say that $I$ is \emph{binuclear} if it is both nuclear and conuclear.
	\end{enumerate}
	We denote by $\nuc(\L)$ the set of binuclear intervals in $\L$. Moreover, we say that $\L$ is a \emph{binuclear lattice} if the nuclear and conuclear intervals of $\L$ coincide.
\end{definition}

As recalled in Theorem~\ref{thm:AP}, Asai and Pfeifer show in \cite{AP} that lattices of torsion classes of abelian length categories are binuclear, and moreover that the binuclear intervals are precisely those whose ``hearts'' are ``wide subcategories'' (see Section~\ref{sec:tors_background} for the definitions). In general, all finite semidistributive lattice (and many infinite ones) are also binuclear, see e.g. Proposition~\ref{prop:ws}. For example, the poset of regions\footnote{By ``the poset of regions'', we mean the poset of regions with respect to an arbitrary choice of base region.} of a simplicial central hyperplane arrangement (from here shortened to just ``poset of regions'') is a semidistributive lattice (see \cite[Theorem~3.4]{BEZ} and \cite[Theorem~3]{reading}), and its binuclear intervals are precisely the ``facial intervals'' considered in \cite[Section~2.3]{DHMP}. For Coxeter arrangements, these are the intervals induced by standard parabolic cosets used to define the facial weak order in \cite[Section~3]{DHP}. Additional examples of binuclear intervals are the following.

\begin{remark}\label{rem:nuclear}
	If $x \covered y \in \L$, then $[x,y]$ is a binuclear interval. Moreover, $[x,x]$ is binuclear for all $x \in \L$.
\end{remark}

We also have the following reformulations of Definition~\ref{def:nuclear}.

\begin{lemma}\label{lem:pop}\
	\begin{enumerate}
		\item Let $x \in \L$, and let $S \subseteq \L$ such that $y \covered x$ (resp. $x \covered y$) for all $y \in S$. Then $[x \wedge (\bigwedge S), x]$ (resp. $[x,x \vee(\bigvee S)]$) is a nuclear (resp. conuclear) interval. Moreover, every nuclear interval of the form $[z,x]$(resp. conuclear interval of the form $[x,z]$) is of this form.
		\item Let $I \subseteq \L$ be an interval. Then $I$ is nuclear (resp. conuclear) if and only if there exists $x \leq I^+$ (resp. $x \geq I^-$) such that $I^- = \pop_x(I^+)$ (resp. $I^+ = \pop^x(I^-)$).
	\end{enumerate}
\end{lemma}

\begin{example}\label{ex:nuclear}
	Let $\L$ be the lattice shown in Figure~\ref{fig:nuclear}. Then the interval $[\hat{0}, \hat{1}]$ is nuclear, but not conuclear since $\pop^{\hat{1}}(\hat{0}) = y \lneq \hat{1}$. Moreover, the set $S$ realizing $\hat{0} = \hat{1} \wedge \left(\bigwedge S\right)$ as in Lemma~\ref{lem:pop}(1) is not unique. Indeed, we have $\hat{0} = \bigwedge \{x,z\} = \bigwedge \{x,y,z\}$.
\end{example}

  \begin{figure}
   	$$\begin{tikzcd}[row sep = 1.5em, column sep = 2em]
		& \hat{1}\arrow[dl]\arrow[dr]\arrow[d]\\
		x \arrow[d]& y\arrow[dl]\arrow[dr] &z\arrow[d]\\
		w\arrow[dr]&&v\arrow[dl]\\
		&\hat{0}&
	\end{tikzcd}$$
       \caption{The lattice in Example~\ref{ex:nuclear}. The interval $[\hat{0}, \hat{1}]$ is nuclear, but not conuclear. Moreover, $\hat{0} = \bigwedge\{x,z\} = \bigwedge \{x,y,z\}$.}\label{fig:nuclear}
    \end{figure}
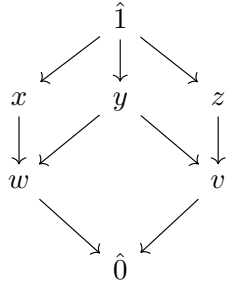

We now give the main definition of this section.

\begin{definition}\label{def:nuc_order}
	Let $\L$ be a complete lattice. We define a relation $\WI$ on $\nuc(\L)$ by $I \WI J$ if and only if $I^- \leq J^-$ and $I^+ \leq J^+$. We refer to $(\nuc(\L),\WI)$, which we often identify with $\nuc(\L)$, as the \emph{binuclear interval order} on $\L$.
\end{definition}

\begin{remark}
	One may also consider a definition analogous to $\WI$ on any set of intervals $\mathcal{I}$ in any poset $\mathcal{P}$. The resulting partial order can then be identified with the subposet of the usual product order on $\mathcal{P} \times \mathcal{P}$ consisting of those elements $(x,y)$ for which $x \leq y$ and $[x,y] \in \mathcal{I}$.
\end{remark}

For example, if $\L$ is a poset of regions, then $(\nuc(\L), \WI)$ can be realized as the ``facial weak order'' as defined in \cite[Section~2.4]{DHMP}. Moreover, for $W$ a Coxeter group, the binuclear intervals in the weak order are those which are induced by the standard parabolic subgroups. Definition~\ref{def:nuc_order} then disagrees with the definition of the weak order given in \cite[Section~3]{DHP}.

To simplify notation, we make the following convention.

\begin{notation}\label{not:nuc_order}
	Let $\L$ be a complete lattice. We denote cover relations in $\nuc(\L)$ by $\WIcover$, and likewise denote $\WIjoin$ and $\WImeet$ whenever the relevant joins and meets exist.
\end{notation}

One of the main results of \cite{DHMP} is that the facial weak order of a poset of regions is a lattice. (This is also shown in \cite{DHP} for Coxeter arrangements and in \cite{KLNPS} in type A.) We prove the analogous result for lattices of torsion classes of $\tau$-tilting finite algebras (see Definition-Theorem~\ref{defthm:g_finite}) in Theorem~\ref{thm:lattice}. For now, we give the following uniform description of those joins and meets in $\nuc(\L)$ which exist.

\begin{lemma}\label{lem:lattice1}
	Let $\L$ be a complete binuclear lattice. Let $I, J \in \nuc(\L)$.
	\begin{enumerate}
		\item If it exists, then $$I \WIjoin J = \left[\pop_{I^- \vee_\L J^-}(I^+ \vee_\L J^+), I^+ \vee_\L J^+\right].$$
		\item If it exists, then $$I \WImeet J = \left[I^- \wedge_\L J^-, \pop^{I^+\wedge_\L J^+}(I^-\wedge_\L J^-)\right].$$
	\end{enumerate}
\end{lemma}

\begin{proof}
	We prove only (1) as the proof of (2) is similar. For readability, denote  $$K := \left[\pop_{I^- \vee_\L J^-}(I^+ \vee_\L J^+), I^+ \vee_\L J^+\right].$$ Note that $K$ is a nuclear interval by Lemma~\ref{lem:pop}, and thus $K \in \nuc(\L)$ by assumption. Moreover, we have that $I, J \WI K$ by construction. Thus it suffices to show that $K$ is minimal in the subposet of common upper bounds of $I$ and $J$. Indeed, suppose $K_0$ is a common upper bound satisfying $K_0 \WI K$. Then in particular $I^+\leq K_0^+ \leq I^+ \vee_\L J^+ = K^+$, and likewise for $J^+$. Thus $K_0^+ = K^+$ by the definition of join. Now $K_0$ is conuclear by assumption, and so $K_0^-$ is the meet of $K_0^+$ with a set of elements covered by $K_0^+ = I^+ \vee J^+$. But, by construction, $K^-$ is the smallest element of this form which lies above $I^- \vee J^-$. We conclude that $K_0^- = K^-$.
\end{proof}

\begin{example}\label{ex:not_lattice}
	Let $\L$ be the lattice shown in Figure~\ref{fig:not_lattice}. Then $[g,a]$, $[h,b]$, $[i,f]$, and $[\hat{0},c]$ are all binuclear intervals. We can then compute $g \wedge_\L h = i$ and $\pop^{a \wedge_\L b}(g \wedge_\L h) = \pop^c(i) = f$. On the other hand, we have $[\hat{0},c] \WI [g,a]$ and $[\hat{0}, c] \WI [h,b]$, but $[\hat{0},c] \not\WI [i,f]$. By Lemma~\ref{lem:lattice1}, this implies that $[g,a]$ and $[h,b]$ do not have a meet in the binuclear interval order.
\end{example}

  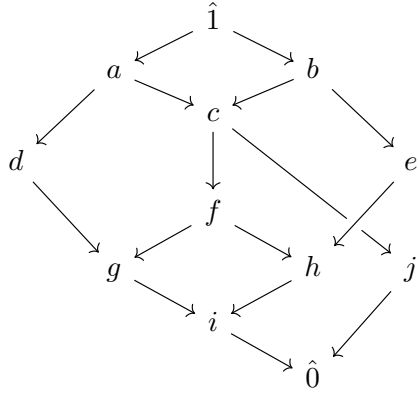
\begin{figure}
   	$$\begin{tikzcd}[row sep = 0.3em, column sep = 2em]
		&& \hat{1}\arrow[dl]\arrow[dr]\\
		&a\arrow[ddl]\arrow[dr] && b\arrow[ddr]\arrow[dl]\\
		&& c\arrow[dd]\arrow[dddrr]\\
		d\arrow[ddr] & & & & e\arrow[ddl,crossing over]\\
		&& f\arrow[dr]\arrow[dl]\\
		&g \arrow[dr] && h\arrow[dl] & j \arrow[ddl]\\
		&& i\arrow[dr]\\
		&&& \hat{0}
	\end{tikzcd}$$
       \caption{The lattice in Example~\ref{ex:not_lattice}. The intervals $[g,a]$ and $[h,b]$ do not have a meet in the binuclear interval order.}\label{fig:not_lattice}
    \end{figure}

In light of this example, we pose the following question.

\begin{question}\label{ques:lattice}
	Let $\L$ be a complete binuclear lattice. What additional properties are necessary and sufficient to ensure that $\nuc(\L)$ is also a lattice?
\end{question}


\section{Torsion classes and wide subcategories}\label{sec:tors_background}

In this section, we recall background information about lattices of torsion classes and wide subcategories of abelian length categories. At the end of this section (Proposition~\ref{prop:restrict_to_interval}(3)), we prove a reduction process for the associated binuclear interval orders.

Let $\A$ be an abelian length category. We will mostly be interested in the case where $\A = \mods A$ is the category of finitely generated (left) modules over a basic finite-dimensional algebra $A$ over some field. In the later case, we denote by $n$ the number of isomorphism classes of simple modules in $\mods A$. By a subcategory of $\A$, we always mean a full subcategory which is closed under isomorphisms. Given a subcategory $\mathcal{C} \subseteq \A$, we associate several additional subcategories as follows. We denote by $\Fac \mathcal{C}$ the subcategory whose objects are quotients of the objects in $\mathcal{C}$ and by $\Gen\mathcal{C}$ the subcategory whose objects quotients of direct sums of (copies of) the objects in $\mathcal{C}$. We denote by $\Filt(\mathcal{C})$ the subcategory consisting of those objects $X$ which admit finite filtrations
$$0 = X_0 \subsetneq X_1 \subsetneq \cdots \subsetneq X_k = X$$
such that $X_i/X_{i-1} \in \mathcal{C}$ for all $i$, and we denote $\tor(\mathcal{C}) := \Filt(\Fac \mathcal{C})$. We also denote
\begin{eqnarray*}
	\mathcal{C}^\perp &:=& \{X \in \mods A \mid \Hom_\A(-,X)|_\mathcal{C} = 0\},\\
	\lperp{\mathcal{C}} &:=& \{X \in \mods A \mid \Hom_\A(X,-)|_\mathcal{C} = 0\}.
\end{eqnarray*}
Moreover, we apply all of the above constructions to a single object $Y$ by taking $\mathcal{C}$ to be the subcategory consisting of all objects isomorphic to $Y$.

We now recall the definition and basic properties of \emph{torsion classes} and \emph{torsion pairs} in $\A$. We refer readers to the expository paper \cite{thomas_intro} and the references therein for further details. We also note that the explicit extension of many of these results from module categories over finite-dimensional algebras to arbitrary abelian length categories can be found in \cite[Section~2]{enomoto_lattice}.

A \emph{torsion pair}, first defined by Dickson \cite{dickson}, is a pair of subcategories $(\T,\F)$ for which $\F = \T^\perp$ and $\T = \lperp{\F}$. If $(\T,\F)$ is a torsion pair, we say that $\T$ is a \emph{torsion class} and that $\F$ is a \emph{torsion free class}. We denote by $\tors \A$ (resp. $\tors A$) the poset of torsion classes of $\A$ (resp. of $\mods A$) with the inclusion relation. It is well-known that $(\T, \T^\perp)$ is a torsion pair, or equivalently that $\T$ is a torsion class, if and only if $\T$ is closed under extensions and quotients. This in particular implies that arbitrary intersections of torsion classes are again torsion classes. Thus, as observed in \cite{IRTT}, $\tors \A$ is a complete lattice whose meet operation is given by intersection. The minimal and maximal elements of this lattice are 0 (the subcategory of objects isomorphic to the zero object) and $\tors \A$.

For every torsion pair $(\T, \F)$ and every $X \in \A$, there exists a unique exact sequence of the form
$$0 \rightarrow X_t \rightarrow X \rightarrow X_f \rightarrow 0$$
with $X_t \in \T$ and $X_f \in \F$. This is often referred to as the ``canonical exact sequence'' of $X$ with respect to $(\T, \F)$.

Given an arbitrary subcategory $\mathcal{C}$, we have that $\lperp{\mathcal{C}}$ and $\tor(\mathcal{C})$ are both torsion classes, and moreover that $\tor(\mathcal{C})$ is the smallest torsion class which contains $\mathcal{C}$. In particular, this means the join operation in $\tors \A$ is given by $\T \vee \U = \tor(\T \cup \U) = \Filt(\T \cup \U)$. It is further shown in \cite[Proposition~3.12]{DIRRT} that an element $\T \in \tors \A$ is compact (resp. cocompact) in the lattice-theoretic sense if and only if there exists a module $X$ such that $\T = \tor(X)$ (resp. $\T = \lperp{X}$). Due to their connection with wide subcategories (defined in the next paragraph), compact torsion classes have also appeared under the name ``widely generated'' in the literature, see e.g. \cite{AP,asai2}.

A subcategory $\W \subseteq \A$ is called a \emph{wide subcategory} if it is closed under extensions, kernels, and cokernels. Equivalently, the wide subcategories of $\A$ are precisely the exact embedded abelian subcategories. Furthermore, every wide subcategory $\W \subseteq \A$ is a length category when considered as an abelian category in its own right. We denote by $\wide \A$ (resp. $\wide A$) the set of wide subcategories of $\A$ (resp. of $\mods A$).

The following explains how the notions of wide subcategories and (bi)nuclear intervals are related. The subcategory $(I^-)^\perp \cap I^+$ in item (3) below is sometimes referred to as the ``heart'' of $I$.

\begin{theorem}\label{thm:AP}\cite[Theorem~1.6]{AP}
	Let $I \subseteq \tors\A$ be an interval. Then the following are equivalent.
	\begin{enumerate}
		\item $I$ is nuclear.
		\item $I$ is conuclear.
		\item $(I^-)^\perp \cap I^+$ is a wide subcategory.
	\end{enumerate}
\end{theorem}

\begin{remark}\label{rem:wide}
	In light of Theorem~\ref{thm:AP}, there is a well-defined map $\mathfrak{W}: \nuc(\tors\A) \rightarrow \wide(\A)$ given by $\mathfrak{W}(I) = (I^-)^\perp \cap I^+$. Moreover, it is shown in \cite[Proposition~6.3]{AP} that the map $\mathfrak{W}$ is surjective. (This also follows from \cite{yurikusa} when $\A = \mods A$ and $|\tors \A| < \infty$.)
\end{remark}

\begin{remark}
	There is also a connection between binuclear intervals and those subcategories of $\mods\Lambda$ which are ``ICE-closed'' (closed under images, extensions, and cokernels). Indeed, it is shown in \cite[Theorem~B]{ES} that the ICE-closed subcategories of $\mods \Lambda$ are precisely those of the form $(I^-)^\perp \cap I^+$ for $I \subseteq \tors\Lambda$ an interval satisfying $I^+ \subseteq \pop^{\hat{1}}(I^-)$. We do not discuss ICE-closed subcategories further in this paper.
\end{remark}

Now, let $\W \subseteq \A$ be a wide subcategory. Since $\W$ is itself an abelian length category, one can define torsion pairs and wide subcategories within $\W$. We note that if $\V$ is a wide subcategory of $\W$, then it will also be a wide subcategory of $\A$. On the other hand, the torsion pairs of $\W$ will in general not be torsion pairs in $\A$. Thus to understand the relationship between the lattices $\tors\A$ and $\tors\W$, we need the following.

\begin{definition}
	Let $\mathcal{C}, \mathcal{D} \subseteq \A$ be subcategories. We denote by $\mathcal{C} *\mathcal{D}$ the full subcategory of $\A$ consisting of those objects $X$ for which there exists a short exact sequence
	$$0 \rightarrow X_C \rightarrow X \rightarrow X_D \rightarrow 0$$
	with $X_C \in \mathcal{C}$ and $X_D \in \mathcal{D}$.
\end{definition}

\begin{proposition}\label{prop:restrict_to_interval}
	Let $I \in \nuc(\tors \A)$. For $\T \in \tors \A$, define $\mathrm{res}_I(\T) = \T \cap \mathfrak{W}(I)$. Then the following hold.
	\begin{enumerate}
		\item \cite[Theorem~1.4]{AP} $\mathrm{res}_I$ restricts to a lattice isomorphism $I \rightarrow \tors(\mathfrak{W}(I)),$
		where $I$ is considered as a sublattice of $\tors\A$. The inverse is given by $\U \mapsto I^- * \U$.
		\item Let $J \in \nuc(\tors\A)$ such that $[I^-,I^-] \WI J \WI [I^+,I^+]$. Then $[\mathrm{res}_I(J^-),\mathrm{res}_I(J^+)] \in \nuc(\tors\mathfrak{W}(I))$.
		\item Let $J \in \nuc(\tors\mathfrak{W}(I))$. Then $[I^- * J^-, I^- * J^+] \in \nuc(\tors\A)$.
		\item $\mathrm{res}_I$ induces a poset isomorphism $$[[I^-,I^-],[I^+,I^+]] \rightarrow (\nuc(\tors \mathfrak{W}(I)), \WI),$$
		where $[[I^-,I^-],[I^+,I^+]] \subseteq \nuc(\tors\A)$ is considered as a subposet under $\WI$.
	\end{enumerate}
\end{proposition}

\begin{proof}
	(2) Let $J \in \nuc(\tors\A)$ such that $[I^-,I^-] \FW J \FW [I^+,I^+]$. Then $I^- \subseteq J^- \subseteq J^+ \subseteq I^+$. Now, we know that $J^-$ is the meet of $J^+$ with a set of elements covered by $J^+$, and since $I^- \subseteq J^-$ each of these elements must lie above $I^-$. Thus $J^-$ can also be realized in this way within the sublattice $I$. The result then follows from (1) and Lemma~\ref{lem:pop}.
	
	(3) Let $J \in \nuc(\tors \mathfrak{W}(I))$. Denote $J^-_I = I^- * J^-$ and $J^+_I = I^- * J^+$. Then $I^- \subseteq J^-_I \subseteq J^+_I \subseteq I^+$ in $\tors \A$. Thus it remains only to show that $[J_I^-,J_I^+] \in \nuc(\tors \A)$. Indeed, we will show that $\mathfrak{W}(J) = (J^-_I)^\perp \cap J^+_I$. The result will then follow from Theorem~\ref{thm:AP}.
	
	Let $X \in (J^-_I)^\perp \cap J^+_I$. Then there is a short exact sequence $0 \rightarrow X^- \rightarrow X \rightarrow X^+ \rightarrow 0$ with $X^+ \in J^+$ and $X^- \in I^-$. Moreover, $X^- \in (J^-_I)^\perp \subseteq (I^-)^\perp$ since $X$ is. Thus $X^- = 0$ and $X \in J^+ \cap (J^-)^\perp$.
	
	Conversely, let $Y \in \mathfrak{W}(J)$. Then $Y \in J^+_I$ since $Y \in J^+$. Similarly, $Y \in (J^-_I)^\perp \subseteq (J^-)^\perp$ by assumption. This proves the result.
	
	(4) This follows immediately from (1), (2), and (3).
\end{proof}

	We will primarily utilize Proposition~\ref{prop:restrict_to_interval} only in the case where $\A = \mods A$ and the torsion classes $I^-$ and $I^+$ are ``functorially finite'' (see Definition-Theorem~\ref{defthm:ff} below). In this special case, we note that Proposition~\ref{prop:restrict_to_interval}(1) can also be found in \cite[Theorem~4.12]{DIRRT} and \cite[Theorem~1.5]{jasso}.


\section{Semistable torsion classes and $g$-vector fans}\label{sec:semistable}

 In this section, we recall the geometric perspective on torsion classes over finite-dimensional algebras. We then introduce the ``facial semistable order'' (Definition~\ref{def:semistable}), which can be seen as the restriction of the binuclear interval order on torsion classes to those intervals which fit into the geometric picture. We refers readers to the more detailed exposition in \cite[Section~2.2]{STTW} for additional background and context. 
 
Throughout this section, we restrict to the case where $\A = \mods A$ for some basic finite-dimensional algebra. We denote the isomorphism classes of simple modules by $\{S(1),\ldots,S(n)\}$ and, for each $i$, by $P(i)$ the projective cover of the module $S(i)$. We denote by $K_0(A)$ the Grothendieck group of $\mods A$, which is isomorphic to the free abelian group with basis $\{[S(i)]\}_{i = 1}^n$. For $X \in \mods A$, we denote by $[X] \in K_0(A)$ the class of $X$. We consider $[X] \in \mathbb{Z}^n$ via the isomorphism $ K_0(A) \rightarrow \mathbb{Z}^n$ given sending $[S(i)]$ to the $i$-th standard basis vector $e_i$. Finally, we denote by $K_0(A)_\mathbb{R} := K_0(A) \otimes_\mathbb{Z} \mathbb{R}$, which we identify with $\mathbb{R}^n$ in the natural way.

Following \cite[Section~3.1]{BKT} and \cite[Section~1.1]{king}, we associate to each $\theta \in \mathbb{R}^n$ the following subcategories.

\begin{definition}\label{def:numerical}
	Let $\theta \in \mathbb{R}^n$. We denote
	\begin{eqnarray*}
		\T_\theta &:=& \{X \in \mods A \mid \theta \cdot [X'] > 0 \text{ for all factor modules }X' \neq 0 \text{ of }X\},\\
		\overline{\T}_\theta &:=& \{X \in \mods A \mid \theta \cdot [X'] \geq 0 \text{ for all factor modules }X' \text{ of }X\},\\
		\F_\theta &:=& \{X \in \mods A \mid \theta \cdot [X'] < 0 \text{ for all submodules }0 \neq X' \subseteq X\},\\
		\overline{\F}_\theta &:=& \{X \in \mods A \mid \theta \cdot [X'] \leq 0 \text{ for all submodules }X' \subseteq X\},\\
		\W_\theta &:=& \overline{\T}_\theta \cap \overline{\F}_\theta.
	\end{eqnarray*}
\end{definition}

These subcategories are related in the following way.

\begin{proposition}\label{prop:numerical}
	Let $\theta \in \mathbb{R}^n$. Then:
	\begin{enumerate}
		\item \cite[Proposition~3.1]{BKT} $(\T_\theta, \overline{\F}_\theta)$ and $(\overline{\T}_\theta, \F_\theta)$ are torsion pairs.
		\item \cite[Proposition~3.24]{BST} $[\T_\theta, \overline{\T}_\theta] \in \nuc(\tors A)$ and $\W_\theta = \mathfrak{W}[\T_\theta, \overline{\T}_\theta]$. In particular, $\W_\theta \in \wide A$.
	\end{enumerate}
\end{proposition}

Using Proposition~\ref{prop:numerical}, we consider the following equivalence relation, defined by Asai in \cite{asai2}.

\begin{definition}\label{def:TF}
	Let $\theta, \eta \in \mathbb{R}^n$. We say that $\theta$ and $\eta$ are \emph{TF-equivalent} if $\T_\theta = \T_\eta$ and $\overline{\T}_\theta = \overline{\T}_\eta$. We denote by $[\theta] \subseteq \mathbb{R}^n$ the TF-equivalence class of a given $\theta \in \mathbb{R}^n$. We denote by $\TF(A) := \{[\theta] \mid \theta \in \mathbb{R}^n\}$. By abuse of notation, for $[\theta] \in \TF(A)$ we denote $\T_{\theta}$ the common value of $\T_v$ for $v \in [\theta]$, are likewise for $\overline{\T}_{\theta}, \F_{\theta}$, and $\overline{\F}_{\theta}$. Finally, we denote $\int[\theta] := [\T_\theta, \overline{\T}_\theta] \in \nuc(\tors A)$.
\end{definition}

Note that if $[\theta] = [\eta]$, then also $\F_\theta = \F_\eta$ and $\overline{\F}_\theta = \overline{\F}_\eta$ by Proposition~\ref{prop:numerical}.

\begin{remark}\label{rem:TF}\
	\begin{enumerate}
		\item The partition of $\mathbb{R}^n$ into $\TF(A)$ is related to another geometric structure, namely the \emph{wall-and-chamber structure} of $A$ defined in \cite{bridgeland,BST,IOTW}. We will not make explicit reference to the wall-and-chamber structure in this paper, but we will use the word ``chamber'' to describe certain TF-equivalence classes (see the discussion following Definition~\ref{def:gfan}). The explicit relationship between TF-equivalence classes and the wall-and-chamber structure is given in \cite[Theorem~1.1]{asai2}.
		\item By Proposition~\ref{prop:numerical}, we have that $\int: \TF(A) \rightarrow \nuc(\tors A)$ is injective. We refer to those intervals in the image of $\int$ as the \emph{semistable intervals} due to their relationship with the stability conditions of \cite{king}.
	\end{enumerate}
\end{remark}

The following properties of TF-equivalence classes will be critical in this paper. We use $(-)^c$ to denote the closure operator in the standard topology on $\mathbb{R}^n$.

\begin{lemma}\cite[Lemmas 2.15 and 2.16]{asai2}\label{lem:closure}
	Let $\theta \in \mathbb{R}^n$. Then
	\begin{enumerate}
		\item $[\theta]$ is convex.
		\item $[\theta]^c = \{\eta \in \mathbb{R}^n \mid \T_\eta \subseteq \T_\theta \subseteq \overline{\T}_\theta \subseteq \overline{\T}_{\eta}\}$.
	\end{enumerate}
\end{lemma}

We now state the main definition of this section.

\begin{definition}\label{def:semistable}
	Let $[\theta], [\eta] \in \TF(A)$. We say that $[\theta] \FW [\eta]$ if $\int[\theta] \WI \int[\eta]$. We refer to $\FW$ as the \emph{facial semistable order} on $\TF(A)$.
\end{definition}

\begin{notation}\label{not:fss_order}
	As in Notation~\ref{not:nuc_order}, we denote cover relations under $\FW$ by $\FWcover$, and likewise denote $\FWjoin$ and $\FWmeet$ whenever the relevant joins and meets exist.
\end{notation}

\begin{remark}\label{rem:nonnegative}
	In \cite[Lemma~2.6]{AI}, Asai and Iyama show that $[\theta] \FW  [\eta]$ whenever $\theta - \eta$ has all nonnegative coefficients. We discuss this further in Remark~\ref{rem:nonnegative} at the end of this section.
\end{remark}

In the sequel, we will also consider the restriction $\FW$ to those equivalence classes which lie in the ``$g$-vector fan'', or alternatively the restriction of $\WI$ to the functorially finite torsion classes. We recall the relevant notions needed to describe these restrictions now.

Denote by $\proj A$ subcategory of $\mods A$ consisting of the projective objects. Denote by $\K^b(\proj A)$ the bounded homotopy category of complexes over $\proj A$, and by $\K^{[-1,0]}(\proj A) \subseteq \K^b(\proj A)$ the full subcategory whose objects are isomorphic to complexes concentrated in degrees 0 and -1. Recall that a basic complex $U \in \K^{[-1,0]}(\proj A)$ is a \emph{2-term presilting complex} if $\Hom_{\K^b(\proj A)}(U,U[1])$, where $[1]$ denotes the shift functor. If in addition $U$ is the direct sum of precisely $n$ indecomposable complexes, we say that $U$ is a \emph{2-term silting complex.}

\begin{definition}
	Let $U \in \K^{[-1,0]}(\proj A)$.
	\begin{enumerate}
		\item Up to isomorphism, we have $U = P^1 \xrightarrow{p} P^0$. Now decompose $P^1 = \bigoplus_{i = 1}^n P(i)^{n_i}$ and $P^0 = \bigoplus_{i = 1}^n P(i)^{m_i}$. Then the \emph{$g$-vector} of $U$ can be identified with
	$$g(U) = (m_i-n_i)\cdot [S(i)] \in K_0(\mods A),$$
	where we have identified $[S(i)]$ with the $i$-th standard basis vector of $\mathbb{R}^n$.
		\item We denote by $C^+(U), C(U) \subseteq \mathbb{R}^n$ the positive (resp. nonnegative) span of the $g$-vectors of the indecomposable direct summands of $U$.
	\end{enumerate}
\end{definition}

For the purposes of this paper, it suffices to take the following as our definition. 

\begin{definition-theorem}\cite[Theorems~2.7 and~3.2]{AIR}\label{defthm:ff}
	Let $\T$ be a torsion class. Then $\T$ is \emph{functorially finite} if the following equivalent conditions hold.
	\begin{enumerate}
		\item There exists a 2-term presilting complex $U$ such that $\T = \tor(H^0(U))$. In this case, we have that $\tor(H^0(U)) = \Gen(H^0(U))$. Moreover, $[g(U)] = C^+(U)$ and $\T = \T_{g(U)}$ by \cite[Proposition~1.3]{asai2}.
		\item There exists a 2-term presilting complex $U$ such that $\T = \lperp{(H^{-1}(\nu U))}$, where $\nu: \proj A \rightarrow \mathsf{inj} A$ is the Nakayama functor. Moreover, in this case we have that $[g(U)] = C^+(U)$ and $\T = \overline{\T}_{g(U)}$ by \cite[Proposition~1.3]{asai2}.
		\item There exists a 2-term silting complex $U$ such that $\T = \tor(H^0(U)) = \lperp{(H^{-1}(\nu U))}$. Moreover, in this case, the complex $U$ is unique up to isomorphism.
	\end{enumerate}
\end{definition-theorem}

\begin{remark}\label{rem:ff}
	One can also state Definition-Theorem~\ref{defthm:ff} in terms of \emph{support $\tau$-tilting pairs} rather than 2-term presilting complexes. These are pairs $(M,P)$ where $M \in \mods A$ satisfies $\Hom(M, \tau M)$ ($\tau$ is the Auslander-Reiten translate) and $P \in \proj A$ satisfies $\Hom(P,M) = 0$. The content of \cite[Theorem~3.2]{AIR} is that there is a bijection between 2-term presilting complexes and support $\tau$-rigid pairs given as follows. Let $U$ be a basic presilting complex. Then one can decompose $U = V \oplus P[1]$, where $P[1]$ is a stalk complex concentrated in degree $-1$ and every nonzero direct summand of $V$ has nonzero homology in degree 0. The bijection then associates $U$ to the pair $(H^0(U),P)$.
\end{remark}

Implicit in Definition-Theorem~\ref{defthm:ff} is that every 2-term presilting complex is a direct summand of a 2-term silting complex. More precisely, we have the following.

\begin{proposition}\cite[Proposition~2.9]{AIR}\label{prop:completions1}
	Let $U$ be a 2-term presilting complex and let $\T \in \tors A$. Then the following are equivalent.
	\begin{enumerate}
		\item There exists a 2-term silting complex $V$ for which (i) $U$ is a direct summand of $V$, and (ii) $\Gen(H^0(V)) = \T = \lperp{(H^{-1}(\nu V))}$.
		\item $\T$ is functorially finite and $\Gen(H^0(U)) \subseteq \T \subseteq \lperp{(H^{-1}(\nu U))}$.
	\end{enumerate}
\end{proposition}

In particular, Proposition~\ref{prop:completions1} provides two distinguished 2-term silting complexes. One is the \emph{Bongartz completion} $B_1(U)$, which is characterized by $\Gen(H^0(B_1(U))) = \Gen(H^0(U))$. The other is the \emph{co-Bongartz completion} $B_0(U)$, which is characterized by $\Gen(H^0(B_0(U))) = \lperp{(H^{-1}(\nu U))}.$

We will also need the following, see \cite[Corollary~6.8]{AP}.

\begin{proposition}\label{prop:ff}
	Let $I \in \nuc(\tors A)$. Then $I^-$ is functorially finite if and only if $I^+$ is functorially finite. Moreover, if both torsion classes are functorially finite, then there exists a 2-term presilting complex $U$ such that $I^- = \T_{g(U)}$ and $I^+ = \overline{\T}_{g(U)}$.
\end{proposition}

We note that it follows from \cite[Proposition~1.3]{asai2} that the complex $U$ in Proposition~\ref{prop:ff} is unique up to isomorphism. In light of the this discussion, we take the following as our definitions.

\begin{definition}\label{def:gfan}\
	\begin{enumerate}
		\item We denote by $\gTF(A)$ the set of TF-equivalence classes $[\theta]$ for which $\T_\theta$ and $\overline{\T}_\theta$ are functorially finite.
		\item The \emph{$g$-vector fan} of $A$ is the collection of the cones $C(U)$ for $U$ a 2-term presilting complex.
	\end{enumerate}
\end{definition}

Given a convex subset $C \subseteq \mathbb{R}^n$, we denote by $\dim(C)$ the dimension of the smallest linear subspace containing $C$. If $\dim(C) = n$, we refer to the interior of $C$ as a \emph{chamber}. It follows from the results of \cite{DIJ} that the $g$-vector fan of $A$ is a simplicial fan. Moreover, the dimension of the cone $C(U)$ is precisely the number of indecomposable direct summands of $C(U)$, and the faces of $C(U)$ correspond to the (not necessarily indecomposable) direct summands. We can then see the chambers which correspond to the Bongartz and co-Bongartz completions of $U$ as being the ``largest'' and ``smallest'' chambers whose closures contain $C(U)$ as a face. In this way, the interval $[\Gen(H^0(U)), \lperp{(H^{-1}(\nu U))}]$ plays the role of the ``facial interval'' associated to a cone in a hyperplane arrangement in \cite{DHMP}. In other words, we have the following description of the restriction of $\FW$ to $\gTF(A)$.

\begin{proposition}[Proposition~\ref{prop:main}]\label{prop:completions}
	Let $U, V$ be 2-term presilting complexes. Then the following are equivalent.
	\begin{enumerate}
		\item $C^+(U) \FW C^+(V)$.
		\item $\Gen(H^0(U)) \subseteq \Gen(H^0(V))$ and $\lperp{(H^{-1}(\nu U))} \subseteq \lperp{(H^{-1}(\nu V))}$.
	\end{enumerate}
\end{proposition}

\begin{remark}\label{rem:no_chambers}

Note that if $\theta \in \mathbb{R}^n$ lies outside the $g$-vector fan, then $\T_\theta \neq \overline{\T}_\theta$ by \cite[Proposition~3.19]{asai2}.  In fact, this proposition is also a key ingredient in the proof of \cite[Theorem~1.4]{asai2}, which shows that there are no chambers (in the wall-and-chamber structure) which lie outside the $g$-vector fan.
Thus there is also no natural analog of ``facial intervals'' for those TF-equivalence classes which lie outside the $g$-vector fan.\end{remark}

When we consider the lattice property of $\nuc(\tors A)$ in Section~\ref{sec:lattice}, we will restrict our attention to algebras whose lattices of torsion classes are finite. In \cite{DIJ} it was shown that these are precisely the \emph{$\tau$-tiliting finite} algebras. In fact, they can be characterized by many equivalent conditions.

\begin{definition-theorem}\label{defthm:g_finite}
	Let $A$ be a finite-dimensional algebra. We say that $A$ is \emph{$\tau$-tilting finite} if $\K^{[-1,0]}(\proj A)$ contains finitely many 2-term presilting complexes (up to isomorphism). This is equivalent to each of the following conditions.
	\begin{enumerate}
		\item \cite[Theorem~1.6]{asai2} The $g$-vector fan of $A$ is complete that is, $\TF(A) = \gTF(A)$.
		\item \cite[Theorem~3.8]{DIJ} $|\tors A| < \infty$.
		\item \cite[Theorem~3.8]{DIJ} Every torsion class in $\tors A$ is functorially finite.
	\end{enumerate}
\end{definition-theorem}

\begin{remark}\label{rem:TF_nuc_finite}
	Note in particular that if $A$ is $\tau$-tilting finite, then the map $\int: (\TF(A),\FW) \rightarrow (\nuc(\tors A), \WI)$ is a poset isomorphism.
\end{remark}

\begin{remark}\label{rem:nonnegative2}
	Recall from Remark~\ref{rem:nonnegative} that Asai and Iyama show that $[\theta] \FW  [\eta]$ whenever $\theta - \eta$ has all nonnegative coefficients. While the converse certainly does not hold, it is an interesting question to determine whether every $\theta, \eta \in \mathbb{R}^n$ such that $[\theta] \FW [\eta]$ admit some $\theta' \in [\theta]$ and $\eta' \in [\eta]$ such that $\theta - \eta$ has all nonnegative coefficients. In particular, the $g$-vector fan of any $\tau$-tilting finite algebra was shown to be polytopal in \cite{fei,AHIKM}, and so the facial semistable order can be seen as a lattice structure on the faces of a polytope. One may then hope to answer this question by studying the type cone of this polytope.
\end{remark}


\section{Cover relations}\label{sec:covers}

In this section, we classify the cover relations in the facial semistable order and its restriction to $\gTF(A)$. Throughout this section, we denote by $A$ a finite-dimensional algebra which is not necessarily $\tau$-tilting finite (unless otherwise stated). 
The following result can be compared with \cite[Proposition~2.14]{DHMP}, which establishes the analogous result for posets of regions.

\begin{lemma}\label{lem:tau_rigid_summand}
	Let $\theta, \eta \in \mathbb{R}^n$. Then 
	\begin{enumerate}
		\item Any two of the following conditions implies the third.
		\begin{enumerate}
			\item $[\theta] \subseteq [\eta]^c$.
			\item $\overline{\T}_\theta = \overline{\T}_\eta$.						\item $[\theta] \FW [\eta]$. 
		\end{enumerate}
		\item Any two of the following conditions implies the third.
		\begin{enumerate}
			\item $[\theta]^c \supseteq [\eta]$.
			\item $\T_\theta = \T_\eta$.
			\item $[\theta] \FW [\eta]$. 
		\end{enumerate}
	\end{enumerate}
\end{lemma}

\begin{proof}
	This follows immediately from the definition and Lemma~\ref{lem:closure}.
\end{proof}

By combining Proposition~\ref{prop:completions} and Lemma~\ref{lem:tau_rigid_summand}, we obtain the following.

\begin{proposition}\label{prop:presilt_cover}
	Let $U$ and $V$ be 2-term presilting complexes.
	\begin{enumerate}
		\item Any two of the following conditions implies the third.
		\begin{enumerate}
			\item $U$ is a direct summand of $V$.
			\item $U$ and $V$ have the same Bongartz completion.
			\item $C^+(U) \FW C^+(V)$.
		\end{enumerate}
		\item Any two of the following conditions implies the third.
		\begin{enumerate}
			\item $V$ is a direct summand of $U$.
			\item $U$ and $V$ have the same co-Bongartz completion.
			\item $C^+(U) \FW C^+(V)$.
		\end{enumerate}
	\end{enumerate}
\end{proposition}

We now establish a necessary condition for the existence of a cover relation.

\begin{proposition}\label{prop:cover1}
	Suppose $[\theta] \FWcover [\eta]$. Then exactly one of the following holds.
	\begin{enumerate}
		\item $[\theta] \subseteq [\eta]^c$ and $\overline{\T}_\theta = \overline{\T}_\eta$.
		\item $[\theta]^c \supseteq [\eta]$ and $\T_\theta = \T_\eta$.
	\end{enumerate}
\end{proposition}

\begin{proof}
	Suppose $[\theta] \FWcover [\eta]$ and define $\gamma: [0,1] \rightarrow \mathbb{R}^n$ by $\gamma(t) = (1-t)\theta + t\eta$. It then follows directly from the definitions that $[\theta] \FW [\gamma(t)] \FW [\eta]$ for all $t$. Since $[\theta]$ and $[\eta]$ are convex (by Lemma~\ref{lem:closure}(1)), the cover relation assumption thus implies that there exists $s \in [0,1]$ such that (i) $\gamma(s) \in [\theta] \cup [\eta]$, (ii) $\gamma(t) \in [\theta]$ for all $t < s$, and (iii) $\gamma(t) \in [\eta]$ for all $t > s$. By Lemma~\ref{lem:closure}(2), this means $[\gamma(s)] \subseteq [\theta]^c \cap [\eta]^c$. It follows that either $[\theta] \subseteq [\eta]^c$ (if $\gamma(s) \in [\theta]$) or $[\theta]^c \supseteq [\eta]$ (if $\gamma(s) \in [\eta]$). Moreover, both properties cannot hold simultaneously without $[\theta] = [\eta]$, which would violate the cover relation assumption. Finally, if $[\theta] \subseteq [\eta]^c$ then $\overline{\T}_\theta = \overline{\T}_\eta$ by Lemma~\ref{lem:tau_rigid_summand}. Similarly, if $[\theta]^c \supseteq [\eta]$ then $\T_\theta = \T_\eta$, again by Lemma~\ref{lem:tau_rigid_summand}.
\end{proof}

In \cite[Proposition~2.15]{DHMP}, the poset-of-regions analog of Proposition~\ref{prop:cover1} is further refined to show that (i) any two faces related by a cover relation in $\WI$ have dimension differing by exactly one, (ii) if two faces satisfy the analog of Proposition~\ref{prop:cover1} and their dimension differs by exactly one, then they are related by a cover relation. The main complication of extending this result to our context is that it is not clear at what level of generality $\TF(A)$ has the structure of a (complete) simplicial fan in $\mathbb{R}^n$. See e.g. \cite[Question~4.5]{fei} and \cite[Section~5]{AI}. On the other hand, since the $g$-vector fan is a simplicial fan, we can prove the following analog of \cite[Proposition~2.15]{DHMP}.

\begin{theorem}[Theorem~\ref{thm:mainA}]\label{thm:covers}
	Let $[\theta], [\eta] \in \gTF(A)$. Then $[\theta] \FWcover [\eta]$ if and only if either (1) $[\theta] \subseteq [\eta]^c$, $\overline{\T}_\theta = \overline{\T}_\eta$, and $\dim[\eta] - \dim[\theta] = 1$, or (b) $[\theta]^c \supseteq [\eta], \T_\theta = \T_\eta,$ and $\dim[\theta] - \dim[\eta] = 1$.
\end{theorem}

\begin{proof}

Let $[\theta], [\eta] \in \gTF(A)$ such that $[\theta] \FWneq [\eta]$.

Suppose first that $[\theta] \FWcover [\eta]$. By Proposition~\ref{prop:cover1}, there are two possibilities. We consider only the case $[\theta] \subseteq [\eta]^c$ and $\overline{\T}_\theta = \overline{\T}_\eta$, the case where $[\theta]^c \supseteq [\eta]$ and $\T_\theta = \T_\eta$ being similar. It then follows from Remark~\ref{rem:ff} that any face of $[\eta]^c$ which contains $[\theta]^c$ is of the form $[\rho]^c$ for some $[\rho] \in \gTF(A)$ which satisfies $[\theta] \FW [\rho] \FW [\eta]$. Since the $g$-vector fan is simplicial, the cover relation assumption thus implies that $\dim[\eta] - \dim[\theta] = 1$.
	
	$(3 \implies 1)$: Suppose that $[\theta]$ and $[\eta]$ do not satisfy (3a) or (3b). If $[\theta] \subseteq [\eta]^c$ and $\overline{\T}_\theta = \overline{\T}_\eta$, this implies that $\dim[\eta] - \dim[\theta] \neq 1$. The fact that $[\theta] \not \FWcover [\eta]$ then follows from the argument in the above paragraph. A similar argument addresses the case where $[\theta]^c \supseteq [\eta]$ and $\T_\theta = \T_\eta.$ The remaining cases then follow from (the contrapositive of) Proposition~\ref{prop:cover1}.
\end{proof}

\begin{remark}
	Note that Theorem~\ref{thm:covers} addresses only cover relations under $\FW$. Thus we cannot yet rule out the possibility that there could exist $[\theta] \FWneq [\eta] \in \gTF(A)$ such that every TF-equivalence class lying between $[\theta]$ and $[\eta]$ lies outside the $g$-vector fan. We will, however, show that such a situation is impossible in Corollary~\ref{cor:covers3}.
\end{remark}

In the statements that follow, we say a continuous function $\gamma: [0,1] \rightarrow \mathbb{R}^n$ is increasing with respect to $\FW$ if $[\gamma(s)] \FW [\gamma(t)]$ for all $0 \leq s \leq t \leq 1$. We note that a similar notion is used in \cite[Definition-Proposition~2.7]{AI}. The following lemma can be compared with \cite[Lemma~2.7]{treffinger}.

\begin{lemma}\label{lem:limit}
	Let $\gamma: [0,1]\rightarrow \mathbb{R}^n$ be increasing with respect to $\FW$. Then for all $s \in [0,1)$, we have $\bigcap_{t > s} \overline{\T}_{\gamma(t)} = \overline{\T}_{\gamma_s}$.
\end{lemma}

\begin{proof}
	The inclusion $\overline{\T}(\gamma_s) \subseteq \bigcap_{t > s} \overline{\T}_{\gamma(t)} $ is clear, so let $X \in \bigcap_{t > s} \overline{\T}_{\gamma(t)}$. Then for any quotient $X'$ of $X$ and any $t > s$, we have $\gamma(t) \cdot \undim X' \geq 0$. Since nonnegativity is a closed condition and $\gamma$ is continuous, it follows that $\theta \cdot \undim X' \geq 0$; i.e., that $X \in \overline{\T}_{\gamma(s)}$.
\end{proof}

We recall the well-known fact that given two torsion classes $\T\subseteq \U \in \tors A$ with $\T$ functorially finite and $\U$ not, there exists a functorially finite torsion class $\V$ with $\T \covered \V \subseteq \U$. See \cite[Theorem~3.14]{jasso} and \cite[Example~3.5]{DIJ}. The following result shows that the behavior is similar in the facial semistable order.

\begin{theorem}[Theorem~\ref{thm:mainB}]\label{thm:covers2}\
	Let $[\theta] \in \gTF(A)$ and $\eta \in \TF(A)$, and suppose that $[\theta] \FWneq [\eta]$. Then there exists $[\rho] \in \gTF(A)$ such that $[\theta] \FWcover [\rho] \FWneq [\eta]$.
\end{theorem}

\begin{proof}
	Define $\gamma: [0,1] \rightarrow \mathbb{R}^n$ by $\gamma(t) = (1-t)\theta + t \eta$. By convexity, we then have two cases to consider.
	
	First suppose there exists $s \in [0,1]$ such that $\gamma(t) \in [\theta]$ if and only if $t < s$. Then $\gamma(s) \in [\theta]^c \setminus [\theta]$. By Lemma~\ref{lem:tau_rigid_summand}, this implies that $\T_\theta = \T_{\gamma(s)}$. It follows as in the proof of Theorem~\ref{thm:covers} that there exists some $[\rho] \in \gTF(A)$ such that $[\theta] \FWcover [\rho] \FW [\gamma(s)] \FW [\eta]$.
	
	Now suppose that there exists $s \in [0,1]$ such that $\gamma(t) \in [\theta]$ if and only if $t \leq s$. Then $\bigcap_{t>s} \overline{\T}_{\gamma(t)} = \overline{T}_{\gamma(s)} = \overline{\T}_\theta$ by Lemma~\ref{lem:limit}. Thus since $\T$ is co-compact by \cite[Proposition~3.2]{DIRRT}, there exists $s < s' \leq 1$ such that $\overline{\T}_{\gamma(s')} = \overline{\T}_\theta$. It again follows that there exists some $[\rho] \in \gTF(A)$ such that $[\theta] \FWcover [\rho] \FW [\gamma(s')] \FW [\eta]$.
\end{proof}

\begin{remark}\label{rem:covers2}\
	\begin{enumerate}
		\item (Theorem~\ref{thm:mainB}) One consequence of Theorem~\ref{thm:covers2} is the following. Suppose $\theta$ lies in the $g$-vector fan and $\eta$ in its complement. If the interior of the line segment connecting $\theta$ and $\eta$ lies entirely outside of the $g$-vector fan, then $[\theta]$ and $[\eta]$ are not related under $\FW$. In particular, $\theta-\eta$ contains both a strictly positive coordinate and a strictly negative coordinate by Remark~\ref{rem:nonnegative}. See Example~\ref{ex:kronecker} for a detailed explanation of this fact for the Kronecker path algebra.
	
	\item We note that the setup described in (1) can only occur when $\theta$ does not lie in a chamber. Indeed, by Remark~\ref{rem:no_chambers} and the discussion preceding Proposition~\ref{prop:completions}, the closure of every chamber is completely contained in the $g$-vector fan. Thus any line segment connecting a point inside a chamber to a point outside the $g$-vector fan will necessarily pass through a point lying in the boundary of the closure of the chamber.
	\end{enumerate}
\end{remark}

We conjecture that the following extension of Theorem~\ref{thm:covers2} is also true.

\begin{conjecture}\label{conj:cover}
	Let $[\theta] \in \gTF(A)$ and $I \in \nuc(I)$ such that $\int[\theta] \WI I$, then there exists $[\rho] \in \gTF(A)$ such that $[\theta] \WIcover [\rho] \WI I$.
\end{conjecture}

\begin{example}\label{ex:kronecker}
	Let $A = \begin{tikzcd} K(1 & 2)\arrow[l,shift left]\arrow[l, shift right]\end{tikzcd}$ be the path algebra of the Kronecker quiver. The $g$-vectors of the indecomposable 2-term presilting complexes are then $(0,-1)$ and those of the form $(-i - 1, i)$ or $(-i + 1, i)$ for $i \in \mathbb{N}$. The $g$-vector fan covers all of $\mathbb{R}^2$ except a single TF-equivalence class, namely the positive ray in the direction $(-1,1)$. See Figure~\ref{fig:kronecker} for a picture, with the equivalence class containing $(-1,1)$ drawn in orange.
	
	Now suppose we are given two vectors $\theta, \eta \in \mathbb{R}^2$ with $\theta$ lying in the $g$-vector fan. Then the only way for the interior of the line segment containing them to lie completely outside the $g$-vector fan is if $\theta = (0,0)$ and $\eta \in [(-1,1)]$. We then observe that $[(0,0)] \not\FW [(-1,1)] $ and vice versa. Indeed, we have that $S(2) \in \T_{(-1,1)} \setminus \T_{(0,0)}$ and $S(1) \in \overline{\T}_{(0,0)} \setminus \overline{\T}_{(-1,1)}$.
\end{example}

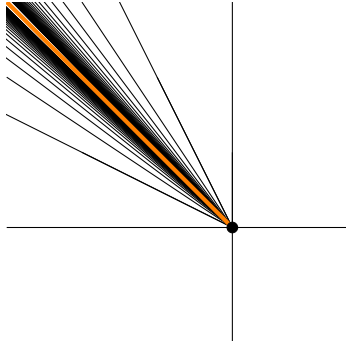
\begin{figure}
\begin{tikzpicture}

\clip (-3,3) rectangle (1.5,-1.5);

\draw (-3,0) -- (3,0);
\draw (0,-3) -- (0,3);
\draw (0,0) -- (-4,2);
\draw (0,0) -- (-2,4);

\foreach \x in {1,...,40}{
	\draw (0,0) -- (-1*\x-1,\x);
	\draw (0,0) -- (-1*\x+1,\x);
	}

\draw[ultra thick,orange] (0,0) -- (-3,3);
	
\filldraw[black] (0,0) circle (2pt) node {};

\end{tikzpicture}
\caption{The $g$-vector fan and its complement for the Kronecker path algebra.}\label{fig:kronecker}
\end{figure}

Combining Theorems~\ref{thm:covers} and~\ref{thm:covers2}, we obtain the following corollary.

\begin{corollary}\label{cor:covers3}
	Let $[\theta], [\eta] \in \gTF(A)$. Then the following are equivalent.
	\begin{enumerate}
		\item $[\theta] \FWcover [\eta]$
		\item $[\theta] \FWneq [\eta]$ and there does not exist $[\rho] \in \gTF(A)$ such that $[\theta] \FWneq [\rho] \FWneq [\eta]$; that is, $[\eta]$ covers $[\theta]$ in the restriction of $\FW$ to $\gTF(A)$.
	\end{enumerate}
\end{corollary}

\begin{remark}\label{rem:inclusion}
	To summarize some of the results of this and the previous section, let $U$  and $U \oplus V$ be 2-term presilting complexes with $V$ indecomposable. Then either $C^+(U) \FWcover C^+(V)$ or $C^+(V) \FWcover C^+(U)$. Moreover, all cover relations in the restriction of $\FW$ to $\gTF(A)$ are of this form. Thus the \emph{undirected} Hasse graph of $(\gTF(A),\FW)$ coincides with that of the usual inclusion order on the cones of the $g$-vector fan, or alternatively with the direct summand order on the set 2-term presilting complexes. The directed Hasse quivers, however, do not coincide. See Figure~\ref{fig:A2} for an explicit example.
\end{remark}


\section{Lattice property}\label{sec:lattice}

In this section, we show that $(\TF(A), \FW) = (\nuc(\tors A), \WI)$ is a lattice when $A$ is $\tau$-tilting finite. Our proof uses similar ideas to that used to prove that the facial weak order of a poset of regions is a lattice in \cite{DHMP}. In particular, our argument is based on the ``BEZ Lemma'', which we recall now.

\begin{lemma}\cite[Lemma~2.1]{BEZ}\label{lem:BEZ}
	Let $\L$ be a finite bounded poset such that the meet $x \wedge y$ exists whenever $x$ and $y$ are both covered by some $z \in \L$. Then $L$ is a lattice.
\end{lemma}

To apply Lemma~\ref{lem:BEZ}, we first need the following reduction result (which is analogous to \cite[Proposition~4.8]{BEZ}). We recall the definition of $\mathrm{res}_J(-)$ from Proposition~\ref{prop:restrict_to_interval}.

\begin{proposition}\label{prop:meet_in_subarrangement}
	Let $\mathcal{A}$ be an abelian length category, and let $I_1, I_2, I_3, J \in \nuc(\tors \mathcal{A})$ such that $J^- \subseteq I_i^- \subseteq I_i^+ \subseteq J^+$ for all $i$. If $\mathrm{res}_J(I_3) = \mathrm{res}_J(I_1) \WImeet \mathrm{res}_J(I_2)$ in the binuclear interval order order of $\tors(\mathfrak{W}(J))$, then $I_3 = I_1 \WImeet I_2$ in the binuclear-interval order of $\tors \A$.
\end{proposition}

\begin{proof}
	Suppose that $\mathrm{res}_J(I_3) = \mathrm{res}_J(I_1) \WImeet \mathrm{res}_J(I_2)$. Then $I_3 \WI I_2$ and $I_3\WI I_2$ by Proposition~\ref{prop:restrict_to_interval}. Thus it remains to show that every $K \in \nuc(\tors\A)$ which satisfies $K \WI I_1$ and $K\WI I_2$ satisfies $K \WI I_3$.
	
	Consider such a $K$, and let $X \in K^-$. Then $X \in I_1^- \cap I_2^-$, and so $X \in J^+$. Thus we have a canonical exact sequence
	$$0 \rightarrow X_J \rightarrow X \rightarrow X' \rightarrow 0$$
	with $X_J \in J^-$ and $X' \in \mathfrak{W}(J)$. Then $X_J \in I_3^-$ because $J^- \subseteq I_3^-$. Since torsion classes are closed under quotients, the assumption that $\mathrm{res}_J(I_3)$ is the meet over $\mathfrak{W}(J)$ implies that $X' \in I_3^- \cap \mathfrak{W}(J)$. Since $I_3^-$ is closed under extensions, we conclude that $X \in I_3^-$, and thus $K^- \subseteq I_3^-$. The argument that $K^+ \subseteq I_3^+$ is analogous.
\end{proof}

We now show that $\nuc(\tors A)$ satisfies the hypotheses of Lemma~\ref{lem:BEZ} when $A$ is $\tau$-tilting finite.

\begin{proposition}\label{thm:BEZ}
	Suppose $A$ is $\tau$-tilting finite. Let $[\theta], [\eta], [\rho] \in \TF(A)$ such that $[\theta] \FWcover [\rho]$ and $[\eta] \FWcover [\rho]$. Then the meet $[\theta] \FWmeet [\eta]$ exists.
\end{proposition}

\begin{proof}
	We have three cases to consider.
	
	Suppose first that both cover relations are of type (2) in Proposition~\ref{prop:cover1}. Then $\T_\theta = \T_\eta$, and so $[\T_\theta, \overline{\T}_\theta \cap \overline{\T}_\eta] \in \nuc(A)$ by Theorem~\ref{thm:AP}. It is straightforward to show that this is the meet.
	
	Next suppose that both cover relations are of type (1) in Proposition~\ref{prop:cover1}. Then $\T_\theta \subseteq \T_\rho \subseteq \overline{\T}_\rho = \overline{\T}_\theta$ and $\T_\eta \subseteq \T_\rho \subseteq \overline{\T}_\rho = \overline{\T}_\eta$. But then $\T_\theta \cap \T_\eta = \pop_{\T_\theta}(\overline{\T}_\rho) \cap \pop_{\T_\eta}(\overline{\T}_\rho) = \pop_{\T_\theta\cap \T_\eta}(\overline{\T}_\rho)$ by Lemma~\ref{lem:pop}. Thus $[\T_\theta \cap \T_\eta, \overline{\T}_\theta] \in \nuc(A)$. It is again straightforward to show that this is the meet.
	
	Finally, suppose that $[\theta] \FWcover [\rho]$ is of type (1) and $[\eta] \FWcover [\rho]$ is of type (2) in Proposition~\ref{prop:cover1}. Then $\T_\eta \subsetneq \T_\rho = \T_\theta$ and $\overline{\T}_\theta \subsetneq \overline{\T}_\rho = \overline{\T}_\eta$. By Propositions~\ref{prop:meet_in_subarrangement} and~\ref{prop:restrict_to_interval}, it suffices to show that the intervals $[0,\W_\eta]$ and $[\T_\theta \cap \W_\eta, \overline{\T}_\theta \cap \W_\eta]$ have a meet in the binuclear interval order of $\W_\eta$. To see this, denote $\U = \pop^{\overline{\T}_\theta \cap \W_\eta}(0) \in \tors \W_\eta$. By construction, $[0,\U]$ is a common lower bound of $[0,\W_\eta]$ and $[\T_\theta \cap \W_\eta, \overline{\T}_\theta \cap \W_\eta]$. Thus let $I \in \nuc(\tors \W_\eta)$ be another common lower bound. Then $I^- = 0$ and $I^+ = \pop^{I^+}(0)$. By construction, this means $I^+ \subseteq \U$. This completes the proof.
\end{proof}

We now conclude the main result of this section.

\begin{theorem}[Theorem~\ref{thm:mainC}]\label{thm:lattice}
	Let $A$ be $\tau$-tilting finite. Then the facial semistable order on $A$ (which is isomorphic to the binuclear interval order on $\tors A$) is a lattice.
\end{theorem}

\begin{proof}
	This follows from Theorem~\ref{thm:BEZ} and Lemma~\ref{lem:BEZ}.
\end{proof}


\section{Semidistributivity}\label{sec:semidistributive}

Reading shows in \cite[Theorem~3]{reading} that the poset of regions of a simplicial central hyperplane arrangement is a (completely) semidistributive lattice (see Definition~\ref{def:kappa_lattice} below). It is then shown in \cite[Theorem~4.27]{DHMP} that this semidistributivity is inherited by the facial weak order. It is likewise shown by Demonet, Iyama, Reading, Reiten, and Thomas in \cite[Theorem~3.1]{DIRRT} that the lattice $\tors A$ is completely semidistributive for any finite-dimensional algebra. (See also \cite[Theorem~4.5]{GM}.) The purpose of this section is to show that, for $\tau$-tilting finite algebras, this semidistributivity similarly passes to the facial semistable order.

As the majority of the results in this section depend only on lattice-theoretic notions, we have opted to state our results in as much generality as possible. Indeed, many of our results hold for infinite lattices, and especially those satisfying the properties related to semidistributivity described in \cite[Section~3.1]{RST} (see Definition~\ref{def:kappa_lattice} below). In particular, while some of our arguments generalize those appearing in \cite[Section~4.3]{DHMP}, the extension of these arguments to the infinite case is new.

We begin by recalling the following definition.

\begin{definition}\label{def:jirr}
	Let $\L$ be a complete lattice and $x \in \L$.
	\begin{enumerate}
		\item We say that $x$ is \emph{completely meet-irreducible} if every $S \subseteq \L$ such that $\bigwedge S = x$ must satisfy $x \in S$. We denote by $\mirr(\L)$ the set of completely meet-irreducible elements in $\L$.
		\item We say that $x$ is \emph{completely join-irreducible} if every $S \subseteq \L$ such that $\bigvee S = x$ must satisfy $x \in S$. We denote by $\jirr(\L)$ the set of completely join-irreducible elements in $\L$.
	\end{enumerate}
\end{definition}

\begin{remark}\label{rem:jirr}
	Let $x \in \L$. Then $x \in \jirr(\L)$ if and only if there exists $x_* \lneq x$ such that every $y \in \L$ for which $y \lneq x$ satisfies $y \leq x_*$. In this case, we have $x_* \covered x$, and this is the unique cover relation of the form $z \covered x$. Dually, $x \in \mirr(\L)$ if and only if there exists $x^* \gneq x$ such that every $y \in \L$ for which $y \gneq x$ satisfies $y \geq x^*$. In this case, we have $x \covered x^*$, and this is the unique cover relation of the form $x \covered z$. For each $j \in \jirr(\L)$ and $m \in \mirr(\L)$, we adopt the common notational convention of denoting by $j_*$ and $m^*$ the elements satisfying these properties.
\end{remark}

\begin{example}\label{ex:BCZ}
	Recall that an object $X \in \A$ is called a \emph{brick} if $\Hom_\A(X,X)$ is a division ring. It is shown in \cite[Theorem~1.5]{BCZ} and \cite[Theorem~1.4]{DIRRT} (see also \cite[Theorem~2.16]{enomoto_lattice} for the explicit extension from module categories to abelian length categories) that $\jirr(\tors \A)$ consists of those torsion classes of the form $\tor(X)$ for $X$ a brick. Moreover, the brick $X$ realizing $\tor(X)$ in this way is unique up to isomorphism. Dually, the torsion classes in $\mirr(\tors \A)$ are those of the form $\lperp{X}$ for $X$ a brick, and this brick is again unique up to isomorphism.
\end{example}

We are now prepared to define the main classes of lattices considered in this section.

\begin{definition}\label{def:kappa_lattice}
	Let $\L$ be a complete lattice.
	\begin{enumerate}
		\item We say that $\L$ is \emph{spatial} if for all $x \in \L$ we have $x = \bigvee \{j \in \L \mid j \leq x \text{ and }j \in \jirr(\L)\}$. The concept of being \emph{co-spatial} is defined dually. We say that $\L$ is \emph{bi-spatial} if it is both spatial and co-spatial.
		\item We say that $\L$ is a \emph{weak meet-$\kappa$-lattice} if it is co-spatial and for all $j \in \jirr(\L)$ there exists a unique $m \in \mirr(\L)$ such that $m \vee j = m^*$ and $m \wedge j = j_*$. The concept of being a \emph{weak join-$\kappa$-lattice} is defined dually.
		\item We say that $\L$ is a \emph{meet-$\kappa$-lattice} if it is co-spatial and for all $j \in \jirr(\L)$ the set $$K(j):=\{y \in \L \mid j \wedge y = j_*\} = \{y \in \L \mid j_* \leq y \text{ and } j \not\leq y\}$$
		has a maximum. The concept of being a \emph{join-$\kappa$-lattice} is defined dually.
		\item We say that $\L$ is \emph{completely meet-semidistributive} if for all $x, y \in \L$ and $S \subseteq \L$ such that $x \wedge s = y$ for all $s \in S$, one has $x \wedge \left(\bigvee S\right) = y$. The concept of being \emph{completely join-semidistributive} is defined dually.
		\item If $\L$ satisfies one of the definitions in (2-4) and its dual, we remove the words meet and join from the name. For example, a weak $\kappa$-lattice is a lattice which is both a weak meet-$\kappa$-lattice and a weak join-$\kappa$-lattice.
	\end{enumerate}
\end{definition}

The notion of semidistributivity dates back to \cite{jonsson}, while the term ``$\kappa$-lattice'' was introduced in the recent paper \cite{RST} in order to generalize the ``fundamental theorem of finite semidistributive lattices'' (Theorem~1.2 of \cite{RST}) to the infinite case. Proposition~\ref{prop:kappa_lattice} below shows how these notions are related. While the majority of this result is not new (See \cite[Proposition~3.10]{RST} for items (3) and (4) and e.g. \cite[Theorem~2.54]{FJN}, \cite[Theorem~2.28]{RST}, or \cite[Theorem~3-1.4]{AN} for portions of the well-known items (5) and (6)), we include a proof for the convenience of the reader.

\begin{proposition}\label{prop:kappa_lattice}
	Let $\L$ be a complete lattice.
	\begin{enumerate}
		\item If $\L$ is a meet-$\kappa$-lattice, then it is also a weak meet-$\kappa$-lattice.
		\item If $\L$ is a join-$\kappa$-lattice, then it is also a weak join-$\kappa$-lattice.
		\item If $\L$ is completely meet-semidistributive and co-spatial, then $\L$ is a meet-$\kappa$-lattice.
		\item If $\L$ is completely join-semidistributive and spatial, then $\L$ is a join-$\kappa$-lattice.
		\item Suppose $\L$ is finite. Then the following are equivalent.
		\begin{enumerate}
			\item $\L$ is a weak meet-$\kappa$-lattice.
			\item $\L$ is a meet-$\kappa$-lattice.
			\item $\L$ is (completely) meet-semidistributive.
		\end{enumerate}
		\item Suppose $\L$ is finite. Then the following are equivalent.
		\begin{enumerate}
			\item $\L$ is a weak join-$\kappa$-lattice.
			\item $\L$ is a join-$\kappa$-lattice.
			\item $\L$ is (completely) join-semidistributive.
		\end{enumerate}
	\end{enumerate}
\end{proposition}

\begin{proof}
	We prove only the odd numbered items, as the even numbered are their duals.
	
	(1) Suppose $\L$ is a meet $\kappa$-lattice, and let $j \in \jirr(\L)$. Denote $m = \max(K(j))$. Then every $y \in \L$ for which $y \gneq m$ satisfies $j \leq y$, and so $m \neq j \vee m \leq y$. It follows that $m \in \mirr(\L)$ and that $j \vee m = m^*$. This shows that the existence portion of Definition~\ref{def:kappa_lattice}(2).
	
	Now suppose for a contradiction that there exists $n \neq m \in \mirr(\L)$ be such that $n \vee j = m^*$ and $n \wedge j = j_*$. Then $n \in K(j)$ by definition. Thus $n \lneq m$, and so $n^* \leq m$. It follows that $j \leq n^* \leq m$, which contradicts that assumption that $j \wedge m = j_*$. We conclude that $\L$ is a weak meet-$\kappa$-lattice.
	
	(3) Suppose that $\L$ is completely meet-semidistributive and co-spatial. Let $j \in \jirr(\L)$. Since $\L$ is co-spatial, there exists $m \in \mirr(\L)$ such that $j_* \leq m$ and $j \not \leq m$, and so $K(j) \neq \emptyset$. Now consider $S \subseteq K(j)$. Then we have $j \wedge \left(\bigvee S\right) = j_*$ by meet-semidistributivity, and so $\bigvee S \in K(j)$. It follows that $K(j)$ has a maximum.

	(5) $(b \implies c)$: Suppose that $\L$ is not meet-semidistributive. Thus there exists $x, y \in \L$ and $S \subseteq \L$ such that $x \wedge s = y$ for all $s \in S$ and $x \wedge\left(\bigvee S\right) \gneq y$. Since $\L$ is finite, we can choose $j \in \L$ which is minimal with respect to the property that $j \leq x \wedge\left(\bigvee S\right)$ and $j \not\leq y$. Now if $z < j$, then $z \leq y$ and so $z \leq j \wedge y \lneq y$. Thus $j \in \jirr(\L)$ and $j \wedge y = j_*$. 
	
	Now let $s \in S$. Since $j \leq x$ and $j \not \leq y = x \wedge s$, it follows that $j \not \leq y$. Thus $j_* = j \wedge y \leq j \wedge s\lneq j$, and so $s \in K(j)$. But $\bigvee S \notin K(j)$ because $j \leq \bigvee S$. It follows that $K(j)$ does not have a maximum, and thus that $\L$ is not a meet-$\kappa$-lattice.
	
	$(c \implies b)$: This is a special case of (3).
	
	$(b \implies a)$: This is a special case of (1).
	
	$(a \implies b)$: Suppose that $\L$ is a weak meet-$\kappa$-lattice, and let $j \in \jirr(\L)$. Let $m \in \mirr(\L)$ be the unique element such that $m \vee j = m^*$ and $m \wedge j = j_*$, and let $x \in K(j)$. Since $\L$ is finite, there exists $x \leq n \in K(j)$ with $n$ maximal in $K(j)$. Moreover, $n \in \mirr(\L)$ by the same argument as in the proof of (1). By the maximality of $n$, we then have $j \leq n^*$, and so $j \wedge n = n^*$. By the assumption, this means that $n = m$, and so $K(j)$ contains a unique maximal element. We conclude that $\L$ is a meet-$\kappa$-lattice.
\end{proof}

The first part of the following corollary follows immediately from the definitions, while the moreover part is implicit in the proof of Proposition~\ref{prop:kappa_lattice}.

\begin{corollary}\label{cor:kappa_lattice}
	A complete lattice $\L$ is a weak $\kappa$-lattice if and only if there exists a unique bijection $\kappa = \kappa_\L: \jirr(\L) \rightarrow \mirr(\L)$ such that $\kappa(j) \vee j = \kappa(j)^*$ and $\kappa(j) \wedge j = j_*$ for all $j \in \jirr(\L)$. Moreover, if $\L$ is a $\kappa$-lattice, then $\kappa(j) = \max(K(j))$ and the dual property holds for $\kappa^{-1}$.
\end{corollary}

Before discussing an example, we also consider the following additional technical property.

\begin{definition}\label{def:ws}
	Let $\L$ be a weak $\kappa$-lattice.
	\begin{enumerate}
		\item We say that $\L$ is \emph{well-separated} $\kappa$ if for all $x \not \leq y \in \L$ there exists $j \in \jirr(\L)$ such that $j \leq x$ and $y \leq \kappa(j)$.
		\item We say $\L$ is well-separated completely semidistributive lattice (abbreviated \emph{ws-csd lattice} if $\L$ is both completely semidistributive and well-separated.
	\end{enumerate}
\end{definition}

Note in particular that every finite semidistributive lattice is well-separated by \cite[Proposition~2.30]{RST}. 

\begin{example}\label{ex:BTZ}
	Consider the setup of Example~\ref{ex:BCZ}. It is shown in \cite[Corollary~8.7]{RST} that $\tors \A$ is a ws-csd lattice. Furthermore,
 it is shown in \cite[Theorem~A]{BTZ} (see also \cite[Theorem~2.16]{enomoto_lattice} for the explicit extension from module categories to abelian length categories) that the map in Corollary~\ref{cor:kappa_lattice} is given by $\kappa(\tor(X)) = \lperp{X}$. \end{example}

\begin{remark}
	By Proposition~\ref{prop:kappa_lattice}, we have that a complete semidistributive lattice $\L$ is a ws-csd lattice if it is bi-spatial and well-separated. That is, it is not necessary to assume directly that $\L$ be a weak $\kappa$-lattice.
\end{remark}

It is convenient to consider the family of ws-csd lattices due to the following.

\begin{proposition}\label{prop:ws}\cite[Corollary~7.9]{BaH2}
	Let $\L$ be a ws-csd lattice. Then $\L$ is binuclear.
\end{proposition}

We now characterize the completely join- and meet-irreducible elements in those binuclear-interval orders which are lattices.

\begin{proposition}\label{prop:jirr}
	Let $\L$ be a complete binuclear lattice for which $\nuc(\L)$ is a lattice. Then
	\begin{enumerate}
		\item $I$ is completely join-irreducible in $\nuc(\L)$ if and only if $I^+$ is completely join-irreducible in $\L$. Moreover, in this case we have either $I^- = I^+$ or $I^- = (I^+)_*$.
		\item $I$ is completely meet-irreducible in $\nuc(\L)$ if and only if $I^-$ is completely meet-irreducible in $\L$. Moreover, in this case $I^+ = I^-$ or $I^+ = (I^-)^*$.
	\end{enumerate}
\end{proposition}

\begin{proof}
	We prove only (1) as the proof of (2) is similar. Let $I \in \nuc(\L)$.
	
	Suppose first that $I^+ \in \jirr(\L)$. Then $[I^+,I^+]$ and $[(I^+)_*, I^+]$ are the only binuclear intervals of the form $[x, I^+]$. This implies the moreover part. We now consider these cases separately.
	
	Suppose first that $I^+ = I^-$. Then for $J \in \nuc(\L)$ we have that $J \WIneq I$ if and only if $J^- \lneq I^-$ and $J^+ \leq I^-$. In particular, $[(I^-)_*,I^-]$ is the unique binuclear interval covered by $I$, and so $I \in \jirr(\nuc(\L))$.
	
	Now suppose that $I^- = (I^+)_*$. Let $J \in \nuc(\L)$ such that $J \WIneq I$. If $J^+ = I^+$, then $J^- = \pop_{J^-}(I^+)$. But $I^-$ and $I^+$ are the only elements of the form $\pop_{(-)}(I^+)$ by the assumption that $I^+$ is completely join-irreducible. It follows that either $J = I$ or $J = [I^-,I^-]$. The other case is that $J^+\lneq I^+$, which is equivalent to $J^+ \leq I^-$. Again, we conclude that $J^- \leq I^-$, and so $J \WI [I^-,I^-]$. This shows that $[I^-,I^-]$ is the unique binuclear interval covered by $I$, and so $I \in \jirr(\nuc(\L))$.
	
	It remains to show that if $I^+ \notin \jirr(\L)$, then $I \notin \jirr(\nuc(\L))$. Suppose first that $I^- = I^+$. Since $I^+ \notin \jirr(\L)$, there exists $S \subseteq \L$ with $I^+ \notin S$ such that $I^+ = \bigvee_\L S$. But then $[s,s] \in \nuc(\L)$ for all $s \in S$. It then follows that $I = \bigvee_{\mathrm{NI}} \{[s,s] \mid s \in S\}$, and so $I \notin \jirr(\nuc(\L))$.
	
	Next suppose $I^- \covered I^+$. It follows that $[I^-,I^-] \WIcover I$. Now since $I^+ \notin \jirr(\L)$, there exists $x \in \L$ such that $x \lneq I^+$ and $x \not \leq I^-$. But then $[x,x] \WIneq I$ and $[x,x] \not\WI [I^-,I^-]$. We conclude that $I \notin \jirr(\L)$.
	
	Finally, suppose $I^- \lneq I^+$, but that $I^- \not\WIcover I^+$. This, together with Lemma~\ref{lem:pop}, implies that there exists $S \subseteq \L$ such that (a) $I^- \covered s$ for all $s \in S$, (b) $I^+  = \bigvee_\L S$, and (c) $|S| \geq 2$. But then $I = \bigvee_{\mathrm{NI}} \{[I^-, s] \mid s \in S\}$. We conclude that $I \notin \jirr(\nuc(\L))$.
\end{proof}

\begin{corollary}\label{cor:jirr} Let $\L$ be a complete binuclear lattice for which $\nuc(\L)$ is a lattice, and let $I \in \nuc(\L)$.
	\begin{enumerate}
		\item Suppose that $I\in\jirr(\nuc(\L))$ and that $I^+ = I^-$. Then $I_* = [(I^+)_*, I^+]$.
		\item Suppose that $I\in\jirr(\nuc(\L))$ and that $I^+ \neq I^-$. Then $I_* = [I^-,I^-]$.
		\item Suppose that $I\in\mirr(\nuc(\L))$ and that $I^+ = I^-$. Then $I^* = [I^+, (I^+)^*]$.
		\item Suppose that $I\in\mirr(\nuc(\L))$ and that $I^+ \neq I^-$. Then $I^* = [I^+,I^+]$.
	\end{enumerate}
\end{corollary}

\begin{theorem}\label{thm:kappa}
	Let $\L$ be a ws-csd lattice for which the binuclear interval order $(\nuc(\L), \WI)$ is a lattice. Then $\nuc(\L)$ is a weak $\kappa$-lattice.
\end{theorem}

\begin{proof}
	For readability, we use $\wedge$ and $\vee$ to mean $\wedge_\L$ and $\vee_\L$ throughout the proof. First note that $\L$ is a weak $\kappa$-lattice by Proposition~\ref{prop:kappa_lattice}. Denote by $\kappa_\L$ the map coming from Corollary~\ref{cor:kappa_lattice}. By Proposition~\ref{prop:jirr}, we can define a bijection $\kappa_{\mathrm{NI}}: \jirr(\nuc(\L)) \rightarrow \mirr(\nuc(\L))$ by
	\begin{eqnarray*}
		[j,j] &\mapsto& [\kappa_\L(j), \kappa_\L(j)^*]\\
		\ [j_*,j] &\mapsto& [\kappa_\L(j),\kappa_\L(j)]
	\end{eqnarray*}
	for all $j \in \jirr(\L)$. We now show that this bijection satisfies the hypotheses of Corollary~\ref{cor:kappa_lattice}. Indeed, for $j \in \jirr(\L)$, Corollary~\ref{cor:jirr} implies that
	\begin{eqnarray*}
		[j,j] \WIjoin [\kappa_\L(j), \kappa_\L(j)^*] &=& [\pop_{j \vee \kappa_\L(j)}(j \vee \kappa_\L(j)^*), j \vee \kappa_\L(j)^*]\\ &=& [\pop_{\kappa_\L(j)^*}(\kappa_\L(j)^*), \kappa_\L(j)^*]\\
		&=& [\kappa_\L(j)^*, \kappa_\L(j)^*]\\ &=& [\kappa_\L(j), \kappa_\L(j)^*]^*, \text{ and } \\
		\ [j_*,j] \WIjoin [\kappa_\L(j),\kappa_\L(j)] &=& [\pop_{\kappa_\L(j)}(\kappa_\L(j)^*), \kappa_\L(j)^*]\\ &=& [\kappa_\L(j), \kappa_\L(j)^*]\\		&=& [\kappa_\L(j), \kappa_\L(j)]^*.
	\end{eqnarray*}
	By duality, this likewise implies that
	\begin{eqnarray*}
		[j,j] \WImeet [\kappa_\L(j), \kappa_\L(j)^*] &=& [j_*, j] = [j,j]_*, \text{ and } \\
		\ [j_*,j] \WImeet [\kappa_\L(j), \kappa_\L(j)] &=& [j_*,j_*] = [j_*,j]_*.
	\end{eqnarray*}
	We conclude that $J \WImeet \kappa_{\mathrm{NI}}(J) = J_*$ and $J \WIjoin \kappa_{\mathrm{NI}}(J) = \kappa_{\mathrm{NI}}(J)^*$ for all $J \in \jirr(\nuc(\L))$. It remains to show that $\kappa_{\mathrm{NI}}$ is the unique bijection with this property. Indeed, let $\lambda: \jirr(\nuc(\L)) \rightarrow \mirr(\nuc(\L)$ be a bijection different from $\kappa_{\mathrm{NI}}$. We then have four possibilities.
	
	Suppose first that there exist $i, j \in \jirr(\L)$ (possibly with $i = j$) such that $\lambda [j, j] = [\kappa_\L(i), \kappa_\L(i)]$. Then \[ [j,j] \WIjoin [\kappa_\L(i), \kappa_\L(i)] = [j \vee \kappa_\L(i) , j \vee \kappa_\L(i)] \neq [\kappa_\L(i), \kappa_\L(i)^*] \] It follows that $\lambda$ does not have the desired property in the case.
	
	Similarly, suppose that there exist $i, j \in \jirr(\L)$ (possibly with $i = j$) such that $\lambda[j_*,j] = [\kappa_\L(i), \kappa_\L(i)^*]$, and suppose for a contradiction that
	\begin{eqnarray*} 
		[j_*,j] \WIjoin [\kappa_\L(i), \kappa_\L(i)^*] &=& [\kappa_\L(i)^*, \kappa_\L(i)^*], \text{ and }\\
		\ [j_*,j] \WImeet [\kappa_\L(i), \kappa_\L(i)^*] &=& [j_*, j_*].
	\end{eqnarray*}
	The first equation implies that $j \leq \kappa_\L(i)^*$, while the second implies that $j_* \leq \kappa(i)$. But then, again by the second equation, we conclude that
	$j_* = \pop^{j}(j_*) = j,$
	a contradiction.
	
	Next, suppose there exist $i \neq j \in \jirr(\L)$ such that $\lambda[j,j] = [\kappa_\L(i), \kappa_\L(i)^*]$. Again suppose for a contradiction that
	\begin{eqnarray*}
		\ [j,j] \WIjoin [\kappa_\L(i),\kappa_\L(i)^*] &=& [\kappa_\L(i)^*, \kappa_\L(i)^*], \text{ and}\\
		\ [j, j] \WImeet [\kappa_\L(i),\kappa_\L(i)^*] &=& [j_*, j].
	\end{eqnarray*}
	The first equation then implies that $j \vee \kappa_\L(i) = \kappa_\L(i)^*$ and the second that $j \wedge \kappa_\L(i) = j_*$. But then $i = j$ by the definition of $\kappa_\L$, which is again a contradiction. The final case, where $i \neq j$ but $\lambda[j_*,j] = [\kappa_\L(i), \kappa_\L(i)]$, is analogous.
	
	To conclude that $\nuc(\L)$ is a weak-$\kappa$ lattice, it remains only to show that it is bi-spatial. To see this, let $I \in \nuc(I)$ and denote $S_1 = \{j \in \jirr(\L) \mid j \leq I^-\}$. Since $\L$ is spatial, we note that $\bigvee_\L S_1 = I^-$. Moreover, by \cite[Lemma~7.1]{BaH2} and the definition of conuclear (c.f. \cite[Definition~7.4]{BaH2}, which agrees with Definition~\ref{def:nuclear} in this case), there exists $S_2 \subseteq \jirr(\L)$ such that (i) $I^+ = I^- \vee \left(\bigvee_\L S_2\right)$, and (ii) $i_* \leq I^-$ for all $i \in S_2$. We can then write
	$$I = \left(\bigvee_{\mathrm{NI}}\left\{[j,j] \mid j \in S_1\right\}\right)\WIjoin\left(\bigvee_{\mathrm{NI}}\left\{[i_*,i] \mid i \in S_2\right\}\right)$$
	as a join of completely join-irreducible elements of $\nuc(\L)$. This shows that $\nuc(\L)$ is spatial. The proof that it is co-spatial is dual.
\end{proof}

The author does not know of any examples where the properties of being well-separated, being a $\kappa$-lattice, and being completely semidistributive fail to pass from a lattice $\L$ to $\nuc(\L)$ (assuming the latter is also a lattice). Thus we pose the following question.

\begin{question}\label{ques:semidistributive}
	Which of the properties defined in Definitions~\ref{def:kappa_lattice} and~\ref{def:ws} are preserved when one passes from $\L$ to $\nuc(\L)$ (assuming that $\nuc(\L)$ is again a lattice)?
\end{question}

By combining Theorem~\ref{thm:kappa} with Proposition~\ref{prop:kappa_lattice}, we immediately conclude the following.

\begin{corollary}\label{cor:semidistributive}
	Let $\L$ be a finite semidistributive lattice. If $\nuc(\L)$ is a lattice, then it is also a finite semidistributive lattice.
\end{corollary}

Combining this with Theorem~\ref{thm:lattice}, we obtain our next main result

\begin{corollary}\label{cor:alg_semidistributive}
	Let $A$ be a $\tau$-tilting finite algbera. Then the facial semistable order on $A$ (which is isomorphic to the binuclear interval order on $\tors A$) is a finite semidistributive lattice.
\end{corollary}


\section{(Sub)lattices associated to wide subcategories}\label{sec:recover_tors}

In this section, we study how the lattice of torsion classes of an abelian length category embeds into the corresponding binuclear interval order. More generally, we show in Theorem~\ref{thm:sublattice} and Remark~\ref{rem:recover_tors} that the binuclear interval order can be partitioned into a set of completely semidistributive lattices.

Throughout this section, fix an arbitrary abelian length category $\mathcal{A}$. In this section, we consider the following sets of wide intervals.

\begin{definition}\label{def:CW}
	Let $\W \in \wide(\mathcal{A})$ be a wide subcategory. Denote 
	\begin{eqnarray*}
		\overline{\C}(\W) &:=& \{I \in \nuc(\tors\mathcal{A}) \mid \W \subseteq \mathfrak{W}(I)\} = \{I \in \nuc(\tors\mathcal{A}) \mid I^- \subseteq \lperp{\W} \text{ and }\W \subseteq I^+\},\\
		\C(\W) &:=& \{I \in \nuc(\tors\mathcal{A}) \mid \W = \mathfrak{W}(I)\}.
	\end{eqnarray*}
\end{definition}

The main result of this section shows that the binuclear intervals in $\C(\W)$ form a completely semidistributive lattice under $\WI$, and that this is a sublattice of $\nuc(\tors \A)$ whenever the latter is a lattice. As a special case, this realizes $\tors \A$ as a subposet (and sublattice where possible) of $\nuc(\tors \A)$. We first note the following.

\begin{lemma}\label{lem:CW}
	Let $\W \in \wide \A$. Then $\overline{\C}(\W)$ is convex with respect to $\WI$.
\end{lemma}

\begin{proof}
	Suppose $I \WI J \WI K$ with $I,K \in \overline{\C}(\W)$. Then $J^- \subseteq K^- \subseteq \lperp{\W}$ and $\W \subseteq I^+ \subseteq J^+$.
\end{proof}

	Note that $\overline{\C}(\W)$ will in general not be an interval. See Example~\ref{ex:A2}.

We also need the following technical lemma.

\begin{proposition}\label{prop:wideInt}
	Let $I \in \nuc(\tors\mathcal{A})$, and let 
	$$0 \rightarrow X \rightarrow Y \rightarrow Z \rightarrow 0$$
	be a short exact sequence with $Y \in I^+$ and $Z \in \mathfrak{W}(I)$. Then $X \in I^+$.
\end{proposition}

\begin{proof}
	By \cite[Proposition~2.12]{IT} (see also \cite[Section~3]{MS}), we have a wide subcategory
	$$\alpha(I^+) = \{M \in I^+ \mid \forall Y \in I^+, \ \forall f:N \rightarrow M: \  \ker f \in I^+\}.$$
	Moreover, by \cite[Theorem~6.6]{AP}, we have that $\mathfrak{W}(I) \subseteq \alpha(I^+)$. (The paper \cite{AP} denotes $\alpha(I^+)$ by $\mathsf{W}_{\mathrm{L}}(I^+)$.) It then follows immediately that $X = \ker(Y \rightarrow Z) \in I^+$.
\end{proof}

We now state the main theorem of this section.

\begin{theorem}[Theorem~\ref{thm:mainD}]\label{thm:sublattice}
	Let $\mathcal{A}$ be an abelian length category and let $\W \in \wide\mathcal{A}$ be a wide subcategory. Then $\C(\W)$ is a completely semidistributive lattice. Moreover, if $\nuc(\tors \mathcal{A})$  is a lattice, then $\C(\W)$ is a sublattice. In particular, taking $\W = 0$ realizes $\tors\mathcal{A}$ as a sublattice of $\nuc(\tors\mathcal{A})$.
\end{theorem}

\begin{proof}
	First note that $\mathcal{C}(\W)$ is nonempty by Remark~\ref{rem:wide}. Let $\mathcal{I} \subseteq \C(\W)$. We consider only meets, as the proof for joins is similar. We will show that $ K := [\cap_{I \in \mathcal{I}} I^-, \cap_{I \in \mathcal{I}} I^+]$ is a binuclear interval by showing that $(K^-)^\perp \cap K^+ = \W$ and applying Theorem~\ref{thm:AP}. This will show that $K = \bigwedge_{\mathrm{NI}} \mathcal{I}$, and that this meet belongs to $\C(\W)$. The inclusion $\W \subseteq (K^-)^\perp \cap K^+$ is immediate since $I^-  \subseteq \lperp{\W}$ and $\W \subseteq I^+$ for all $I \in \mathcal{I}$. Thus let $X \in (K^-)^\perp \cap K^+$. Then there is a filtration
	$$0 = X_0 \subseteq \cdots \subseteq X_k = X$$
such that  each $X_i/X_{i-1} \in \bigcup_{I \in \mathcal{I}} (I^-)^\perp$. We will show that $X_i/X_{i-1} \in \W$ and that $X_i \in (K^-)^\perp \cap K^+$ for all $i$ by reverse induction on $i$. 

For the base case, note that there exists $I \in \mathcal{I}$ such that $X_k/X_{k-1} \in (I^-)^\perp$. Since torsion classes are closed under quotients, we have $X_k/X_{k-1} \in I^+$, and so $X_k /X_{k-1} \in \W$. By Proposition~\ref{prop:wideInt}, this implies that $X_{k-1} \in J^+$ for all $J \in \mathcal{I}$, and so $X_{k-1} \in K^+$. The induction step then follows using the same argument. We conclude that $X \in \W$, as needed.
\end{proof}

While Theorem~\ref{thm:sublattice} shows that $\C(\W)$ is a lattice, Example~\ref{ex:A3} shows that it may not be isomorphic to the lattice of torsion classes of any abelian length category.

\begin{remark}\label{rem:recover_tors}
	As a consequence of the proof of Theorem~\ref{thm:sublattice}, we have that $\mathcal{C}(\W)$ is a sublattice of the product lattice $\tors \A \times \tors \A$; that is, meets and joins in $\mathcal{C}(\W)$ are computed by taking the meets and joins of the tops and bottoms of the intervals in $\tors \A$ separately. In particular, since $\tors \A$ is completely semidistributive, this also implies that $\mathcal{C}(\W)$ is completely semidistributive.
\end{remark}

\begin{remark}\label{rem:lattice_recover_tors}
It is an interesting problem to give a lattice-theoretic interpretation of Theorem~\ref{thm:sublattice}. For example, if $\L$ is a ws-csd lattice, then the binuclear intervals of $\L$ can be partitioned according to their set of ``atoms'', see e.g. \cite[Section~7.1]{BaH2}. For lattices of torsion classes, it then follows from \cite[Theorem~1.4(3)]{AP} that this corresponds (under the bijections in Example~\ref{ex:BCZ}) to partioning the set of binuclear intervals into the sets $\mathcal{C}(\W)$. One could then use this correspondence to look for an alternate proof of Theorem~\ref{thm:sublattice} which no longer relies on arguments which are specific to torsion classes.
\end{remark}


\section{Examples}\label{sec:examples}

\begin{example}\label{ex:A2}
Consider the path algebra $K(1\leftarrow 2)$. This path algebra has only three indecomposable modules: $S(1) = P(1), S(2)$, and $P(2)$. Each of these is ($\tau$-)rigid, and they fit into an exact sequence $0 \rightarrow S(2) \rightarrow P(1) \rightarrow S(1) \rightarrow 0$. (Up to scalar multiplication, there are no non-identity morphisms between indecomposables other than those which appear in this short exact sequence.)  Then the lattice of torsion classes is:
$$
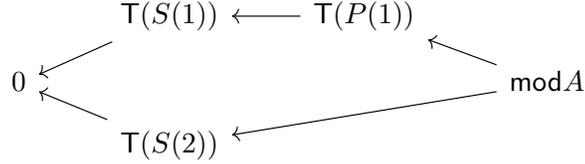
\begin{tikzcd}[row sep = 0.25cm]
	&\tor(S(1))\arrow[dl] & \tor(P(1))\arrow[l]\\
	0 &&& \mods A\arrow[ul]\arrow[dll]\\
	& \tor(S(2))\arrow[ul]
\end{tikzcd}$$

The support $\tau$-rigid pairs and corresponding binuclear intervals in $\tors A$ are then as follows:

\renewcommand{\arraystretch}{1.5}

\begin{center}
	\begin{tabular}{|c|c||c|c|}
		\hline
		support $\tau$-tilting pair & binuclear interval & support $\tau$-tilting pair & binuclear interval\\
		\hline
		\hline
		$(0,P(1)\oplus P(2))$ & $[0,0]$ & $(0,P(2))$ & $[0,\tor(P(1))]$\\
		\hline
		$(P(1),P(2))$ & $[\tor(P(1)),\tor(P(1))]$ & $(P(1),0)$ & $[\tor(P(1)),\mods A]$\\
		\hline
		$(0,0)$ & $[0,\mods A]$ & (0,P(1)) & $[0,\tor(S(2))]$\\
		\hline
		$(S(2),P(1))$ & $[\tor(S(2)),\tor(S(2))]$ & (S(2),0) & $[\tor(S(2)),\tor(P(2))]$\\
		\hline
		$(P(2) \oplus S(2),0)$ & $[\tor(P(2)),\tor(P(2))]$ & $(P(2),0)$ & $[\tor(P(2)),\mods A]$\\
		\hline
		$(P(1) \oplus P(2),0)$ & $[\mods A,\mods A]$ &&\\
		\hline
	\end{tabular}
\end{center}

The facial semistable order is shown in Figure~\ref{fig:A2}, with each TF-equivalence class labeled by its corresponding $\tau$-rigid pair. The five colored regions represent the five chambers, each of which is labeled by a support $\tau$-tilting pair (i.e., a pair with two indecomposable summands.) The rays where pairs of chambers meet are each TF-equivalence classes labeled by the common direct summand of the two adjacent chambers. Finally, the origin is its own TF-equivalence class.

\begin{figure}
{\small
$$\begin{tikzcd}[column sep = 0.5cm,every matrix/.style={name=mymatr},
    execute at end picture={
    	\draw[color=blue,fill,opacity=.25]
		([xshift=-1.5cm]mymatr-4-1.center) -- (mymatr-4-4.center) -- ([xshift=-1.5cm,yshift=0.5cm]mymatr-4-1.center|-mymatr-1-3.center) -- cycle;
  	  \draw[color=orange,fill,opacity=.25]
		(mymatr-4-4.center) -- ([xshift=-1.5cm,yshift=0.5cm]mymatr-4-1.center|-mymatr-1-3.center) -- ([yshift=0.5cm]mymatr-1-4.center) -- cycle;
	\draw[color=blue,fill,opacity=.25]
		(mymatr-4-4.center) rectangle ([xshift=1.5cm,yshift=0.5cm]mymatr-1-5.center);
    	\draw[color=orange,fill,opacity=.25]
		(mymatr-4-4.center) rectangle ([xshift=1.5cm,yshift=-0.5cm]mymatr-5-5.center);
	\draw[color=teal,fill,opacity=.25]
		(mymatr-4-4.center) rectangle ([xshift=-1.5cm,yshift=-0.5cm]mymatr-5-1.center);
    }]
	&&(P(2) \oplus S(2),0) \arrow[dl,thick] & (P(2),0) \arrow[ddd,thick]\arrow[l,thick] & (P(1)\oplus P(2),0)\arrow[l,thick]\arrow[ddd,thick]\\
	& (S(2),0) \arrow[dl,thick]\\
	(S(2), P(1))\arrow[d,thick]\\
	(0,P(1)) \arrow[d,thick]&&&(0,0)\arrow[lll,thick]\arrow[d,thick] & (P(1),0)\arrow[l,thick]\arrow[d,thick]\\
	(0, P(1) \oplus P(2)) &&&(0,P(2))\arrow[lll,thick]&(P(1),P(2))\arrow[l,thick]
\end{tikzcd}$$
}
\caption{The facial semistable order for the path algebra $K(1\leftarrow 2)$.}\label{fig:A2}
\end{figure}
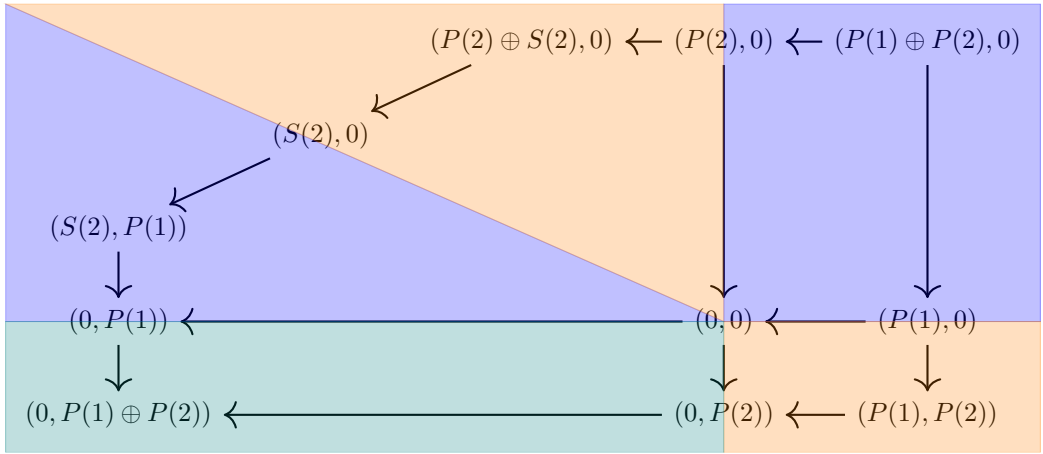

Now consider the wide subcategory $\W = \add(P(2))$, which consists of all objects isomorphic to direct sums of copies of $P(2)$. Then $\overline{\C}(\W) = \{[0,\mods A], [\tor(S(2)),\tor(P(2))]\}$, and these are incomparable under $\WI$. This shows that $\C(\W)$, which is convex by Lemma~\ref{lem:CW}, is not an interval.
\end{example}

\begin{example}\label{ex:A3}
	Consider the path algebra $K(1 \rightarrow 2 \leftarrow 3)$. The Auslander-Reiten quiver of this algebra is:
	$$
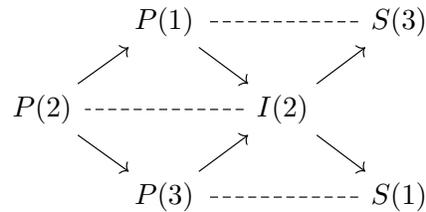
\begin{tikzcd}[row sep = 0.5cm, column sep = 0.5cm]
		&P(1) \arrow[dr]&&S(3)\arrow[dashed,no head,ll]\\
		P(2)\arrow[ur]\arrow[dr] && I(2)\arrow[ur]\arrow[dr]\arrow[dashed,no head,ll]\\
		&P(3)\arrow[ur]&&S(1)\arrow[dashed,no head,ll]
	\end{tikzcd}$$
	where $I(2)$ is the injective envelope of the simple module $S(2)$ and has $\undim I(2) = (1,1,1)$.
	The lattice of torsion classes is shown in Figure~\ref{fig:A3}.
	
	\begin{figure}
		\begin{tikzcd}[row sep = 0.5cm,column sep = 0.25cm]
			&&\mods A\arrow[d]\arrow[ddll]\arrow[ddrr]\\
			&&\lperp{P(2)}\arrow[dl]\arrow[dr]\\
			\lperp{S(1)}\arrow[dd]\arrow[dddrr] & \lperp{P(1)}\arrow[ddl,dashed]\arrow[dr]&&\lperp{P(3)}\arrow[dl]\arrow[ddr,dashed]& \lperp{S(3)}\arrow[dd]\arrow[dddll]\\
			&&\tor(I(2))\arrow[d]&& \\
			\tor(P(3))\arrow[dr]&& \lperp{I(2)}\arrow[dr,dashed]\arrow[dl,dashed]&&\tor(P(1))\arrow[dl]\\
			&\tor(S(3))\arrow[dr] & \tor(P(2))\arrow[d] & \tor(S(1))\arrow[dl]\\
			&&0
		\end{tikzcd}
\caption{The lattice of torsion classes of the path algebra $K(1\rightarrow 2\leftarrow 3)$.}\label{fig:A3}
	\end{figure}
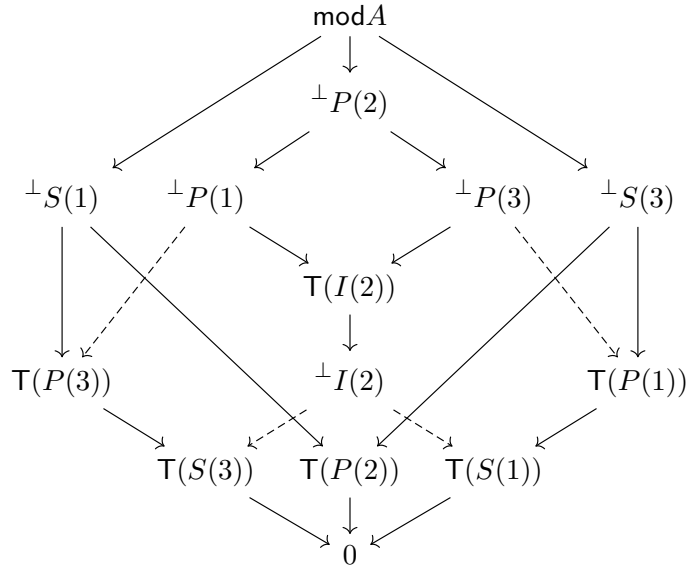
	Now let $\W = \add(P(1))$ be the wide subcategory which consists of all objects isomorphic to direct sums of copies of $P(1)$. One can then show that $\C(\W)$ consists of precisely three elements which satisfy $$\left[\lperp{P(1)},\lperp{P(2)}\right] \WI \left[\T(I(2)),\lperp{P(3)}\right] \WI \left[\T(S(1)),\T(P(1))\right].$$
	In particular, the lattice $\C(\W)$ is not isomorphic to the lattice of torsion classes of a finite-dimensional algebra since it is not Hasse regular.
	
\end{example}


\bibliographystyle{amsalpha}
\bibliography{biblio.bib}

\end{document}